\documentclass[11pt,leqno]{amsart}
\usepackage{amssymb}
\usepackage{xypic}
\begin{document}
\theoremstyle{plain}
\newtheorem{thm}{Theorem}[section]
\newtheorem*{thm1}{Theorem 1}
\newtheorem*{thm2}{Theorem 2}
\newtheorem{lemma}[thm]{Lemma}
\newtheorem{lem}[thm]{Lemma}
\newtheorem{cor}[thm]{Corollary}
\newtheorem{prop}[thm]{Proposition}
\newtheorem{propose}[thm]{Proposition}
\newtheorem{variant}[thm]{Variant}
\theoremstyle{definition}
\newtheorem{notations}[thm]{Notations}
\newtheorem{rem}[thm]{Remark}
\newtheorem{rmk}[thm]{Remark}
\newtheorem{rmks}[thm]{Remarks}
\newtheorem{defn}[thm]{Definition}
\newtheorem{ex}[thm]{Example}
\newtheorem{claim}[thm]{Claim}
\newtheorem{ass}[thm]{Assumption}
\numberwithin{equation}{section}
\newcounter{elno}                
\def\points{\list
{\hss\llap{\upshape{(\roman{elno})}}}{\usecounter{elno}}} 
\let\endpoints=\endlist


\catcode`\@=11
%
%
\def\opn#1#2{\def#1{\mathop{\kern0pt\fam0#2}\nolimits}} 
\def\bold#1{{\bf #1}}%
\def\underrightarrow{\mathpalette\underrightarrow@}
\def\underrightarrow@#1#2{\vtop{\ialign{$##$\cr
 \hfil#1#2\hfil\cr\noalign{\nointerlineskip}%
 #1{-}\mkern-6mu\cleaders\hbox{$#1\mkern-2mu{-}\mkern-2mu$}\hfill
 \mkern-6mu{\to}\cr}}}
\let\underarrow\underrightarrow
\def\underleftarrow{\mathpalette\underleftarrow@}
\def\underleftarrow@#1#2{\vtop{\ialign{$##$\cr
 \hfil#1#2\hfil\cr\noalign{\nointerlineskip}#1{\leftarrow}\mkern-6mu
 \cleaders\hbox{$#1\mkern-2mu{-}\mkern-2mu$}\hfill
 \mkern-6mu{-}\cr}}}
%
%

%
\def\:{\colon}
\let\oldtilde=\tilde
\def\tilde#1{\mathchoice{\widetilde{#1}}{\widetilde{#1}}%
{\indextil{#1}}{\oldtilde{#1}}}
\def\indextil#1{\lower2pt\hbox{$\textstyle{\oldtilde{\raise2pt%
\hbox{$\scriptstyle{#1}$}}}$}}
\def\pnt{{\raise1.1pt\hbox{$\textstyle.$}}}
%

%
\let\amp@rs@nd@\relax
\newdimen\ex@\ex@.2326ex
\newdimen\bigaw@
\newdimen\minaw@
\minaw@16.08739\ex@
\newdimen\minCDaw@
\minCDaw@2.5pc
\newif\ifCD@
\def\minCDarrowwidth#1{\minCDaw@#1}
\newenvironment{CD}{\@CD}{\@endCD}
\def\@CD{\def\A##1A##2A{\llap{$\vcenter{\hbox
 {$\scriptstyle##1$}}$}\Big\uparrow\rlap{$\vcenter{\hbox{%
$\scriptstyle##2$}}$}&&}%
\def\V##1V##2V{\llap{$\vcenter{\hbox
 {$\scriptstyle##1$}}$}\Big\downarrow\rlap{$\vcenter{\hbox{%
$\scriptstyle##2$}}$}&&}%
\def\={&\hskip.5em\mathrel
 {\vbox{\hrule width\minCDaw@\vskip3\ex@\hrule width
 \minCDaw@}}\hskip.5em&}%
\def\verteq{\Big\Vert&&}%
\def\noarr{&&}%
\def\vspace##1{\noalign{\vskip##1\relax}}\relax\let\amp@rs@nd@&\iffalse}\fi
 \CD@true\vcenter\bgroup\relax\let\\=\cr\iffalse}\fi\tabskip\z@skip\baselineskip20\ex@
 \lineskip3\ex@\lineskiplimit3\ex@\halign\bgroup
 &\hfill$\m@th##$\hfill\cr}
\def\@endCD{\cr\egroup\egroup}
%
\def\>#1>#2>{\amp@rs@nd@\setbox\z@\hbox{$\scriptstyle
 \;{#1}\;\;$}\setbox\@ne\hbox{$\scriptstyle\;{#2}\;\;$}\setbox\tw@
 \hbox{$#2$}\ifCD@
 \global\bigaw@\minCDaw@\else\global\bigaw@\minaw@\fi
 \ifdim\wd\z@>\bigaw@\global\bigaw@\wd\z@\fi
 \ifdim\wd\@ne>\bigaw@\global\bigaw@\wd\@ne\fi
 \ifCD@\hskip.5em\fi
 \ifdim\wd\tw@>\z@
 \mathrel{\mathop{\hbox to\bigaw@{\rightarrowfill}}\limits^{#1}_{#2}}\else
 \mathrel{\mathop{\hbox to\bigaw@{\rightarrowfill}}\limits^{#1}}\fi
 \ifCD@\hskip.5em\fi\amp@rs@nd@}
\def\<#1<#2<{\amp@rs@nd@\setbox\z@\hbox{$\scriptstyle
 \;\;{#1}\;$}\setbox\@ne\hbox{$\scriptstyle\;\;{#2}\;$}\setbox\tw@
 \hbox{$#2$}\ifCD@
 \global\bigaw@\minCDaw@\else\global\bigaw@\minaw@\fi
 \ifdim\wd\z@>\bigaw@\global\bigaw@\wd\z@\fi
 \ifdim\wd\@ne>\bigaw@\global\bigaw@\wd\@ne\fi
 \ifCD@\hskip.5em\fi
 \ifdim\wd\tw@>\z@
 \mathrel{\mathop{\hbox to\bigaw@{\leftarrowfill}}\limits^{#1}_{#2}}\else
 \mathrel{\mathop{\hbox to\bigaw@{\leftarrowfill}}\limits^{#1}}\fi
 \ifCD@\hskip.5em\fi\amp@rs@nd@}
%
%
\newenvironment{CDS}{\@CDS}{\@endCDS}
\def\@CDS{\def\A##1A##2A{\llap{$\vcenter{\hbox
 {$\scriptstyle##1$}}$}\Big\uparrow\rlap{$\vcenter{\hbox{%
$\scriptstyle##2$}}$}&}%
\def\V##1V##2V{\llap{$\vcenter{\hbox
 {$\scriptstyle##1$}}$}\Big\downarrow\rlap{$\vcenter{\hbox{%
$\scriptstyle##2$}}$}&}%
\def\={&\hskip.5em\mathrel
 {\vbox{\hrule width\minCDaw@\vskip3\ex@\hrule width
 \minCDaw@}}\hskip.5em&}
\def\verteq{\Big\Vert&}
\def\novarr{&}
\def\noharr{&&}
\def\SE##1E##2E{\slantedarrow(0,18)(4,-3){##1}{##2}&}
\def\SW##1W##2W{\slantedarrow(24,18)(-4,-3){##1}{##2}&}
\def\NE##1E##2E{\slantedarrow(0,0)(4,3){##1}{##2}&}
\def\NW##1W##2W{\slantedarrow(24,0)(-4,3){##1}{##2}&}
\def\slantedarrow(##1)(##2)##3##4{%
\thinlines\unitlength1pt\lower 6.5pt\hbox{\begin{picture}(24,18)%
\put(##1){\vector(##2){24}}%
\put(0,8){$\scriptstyle##3$}%
\put(20,8){$\scriptstyle##4$}%
\end{picture}}}
\def\vspace##1{\noalign{\vskip##1\relax}}\relax\let\amp@rs@nd@&\iffalse}\fi
 \CD@true\vcenter\bgroup\relax\let\\=\cr\iffalse}\fi\tabskip\z@skip\baselineskip20\ex@
 \lineskip3\ex@\lineskiplimit3\ex@\halign\bgroup
 &\hfill$\m@th##$\hfill\cr}
\def\@endCDS{\cr\egroup\egroup}
%
\newdimen\TriCDarrw@
\newif\ifTriV@
\newenvironment{TriCDV}{\@TriCDV}{\@endTriCD}
\newenvironment{TriCDA}{\@TriCDA}{\@endTriCD}
\def\@TriCDV{\TriV@true\def\TriCDpos@{6}\@TriCD}
\def\@TriCDA{\TriV@false\def\TriCDpos@{10}\@TriCD}
\def\@TriCD#1#2#3#4#5#6{%
\setbox0\hbox{$\ifTriV@#6\else#1\fi$}
\TriCDarrw@=\wd0 \advance\TriCDarrw@ 24pt
\advance\TriCDarrw@ -1em
\def\SE##1E##2E{\slantedarrow(0,18)(2,-3){##1}{##2}&}
\def\SW##1W##2W{\slantedarrow(12,18)(-2,-3){##1}{##2}&}
\def\NE##1E##2E{\slantedarrow(0,0)(2,3){##1}{##2}&}
\def\NW##1W##2W{\slantedarrow(12,0)(-2,3){##1}{##2}&}

\def\slantedarrow(##1)(##2)##3##4{\thinlines\unitlength1pt
\lower 6.5pt\hbox{\begin{picture}(12,18)%
\put(##1){\vector(##2){12}}%
\put(-4,\TriCDpos@){$\scriptstyle##3$}%
\put(12,\TriCDpos@){$\scriptstyle##4$}%
\end{picture}}}
\def\={\mathrel {\vbox{\hrule
   width\TriCDarrw@\vskip3\ex@\hrule width
   \TriCDarrw@}}}
\def\>##1>>{\setbox\z@\hbox{$\scriptstyle
 \;{##1}\;\;$}\global\bigaw@\TriCDarrw@
 \ifdim\wd\z@>\bigaw@\global\bigaw@\wd\z@\fi
 \hskip.5em
 \mathrel{\mathop{\hbox to \TriCDarrw@
{\rightarrowfill}}\limits^{##1}}
 \hskip.5em}
\def\<##1<<{\setbox\z@\hbox{$\scriptstyle
 \;{##1}\;\;$}\global\bigaw@\TriCDarrw@
 \ifdim\wd\z@>\bigaw@\global\bigaw@\wd\z@\fi
 \mathrel{\mathop{\hbox to\bigaw@{\leftarrowfill}}\limits^{##1}}
 }
 \CD@true\vcenter\bgroup\relax\let\\=\cr\iffalse}\fi
 \tabskip\z@skip\baselineskip20\ex@
 \lineskip3\ex@\lineskiplimit3\ex@
 \ifTriV@
 \halign\bgroup
 &\hfill$\m@th##$\hfill\cr
#1&\multispan3\hfill$#2$\hfill&#3\\
&#4&#5\\
&&#6\cr\egroup%
\else
 \halign\bgroup
 &\hfill$\m@th##$\hfill\cr
&&#1\\%
&#2&#3\\
#4&\multispan3\hfill$#5$\hfill&#6\cr\egroup
\fi}
\def\@endTriCD{\egroup} 
\newcommand{\mc}{\mathcal} 
\newcommand{\mb}{\mathbb} 
\newcommand{\surj}{\twoheadrightarrow} 
\newcommand{\inj}{\hookrightarrow} \newcommand{\zar}{{\rm zar}} 
\newcommand{\an}{{\rm an}} \newcommand{\red}{{\rm red}} 
\newcommand{\Rank}{{\rm rk}} \newcommand{\codim}{{\rm codim}} 
\newcommand{\rank}{{\rm rank}} \newcommand{\Ker}{{\rm Ker \ }} 
\newcommand{\Pic}{{\rm Pic}} \newcommand{\Div}{{\rm Div}} 
\newcommand{\Hom}{{\rm Hom}} \newcommand{\im}{{\rm im}} 
\newcommand{\Spec}{{\rm Spec \,}} \newcommand{\Sing}{{\rm Sing}} 
\newcommand{\sing}{{\rm sing}} \newcommand{\reg}{{\rm reg}} 
\newcommand{\Char}{{\rm char}} \newcommand{\Tr}{{\rm Tr}} 
\newcommand{\Gal}{{\rm Gal}} \newcommand{\Min}{{\rm Min \ }} 
\newcommand{\Max}{{\rm Max \ }} \newcommand{\Alb}{{\rm Alb}\,} 
\newcommand{\GL}{{\rm GL}\,} 
\newcommand{\ie}{{\it i.e.\/},\ } \newcommand{\niso}{\not\cong} 
\newcommand{\nin}{\not\in} 
\newcommand{\soplus}[1]{\stackrel{#1}{\oplus}} 
\newcommand{\by}[1]{\stackrel{#1}{\rightarrow}} 
\newcommand{\longby}[1]{\stackrel{#1}{\longrightarrow}} 
\newcommand{\vlongby}[1]{\stackrel{#1}{\mbox{\large{$\longrightarrow$}}}} 
\newcommand{\ldownarrow}{\mbox{\Large{\Large{$\downarrow$}}}} 
\newcommand{\lsearrow}{\mbox{\Large{$\searrow$}}} 
\renewcommand{\d}{\stackrel{\mbox{\scriptsize{$\bullet$}}}{}} 
\newcommand{\dlog}{{\rm dlog}\,} 
\newcommand{\longto}{\longrightarrow} 
\newcommand{\vlongto}{\mbox{{\Large{$\longto$}}}} 
\newcommand{\limdir}[1]{{\displaystyle{\mathop{\rm lim}_{\buildrel\longrightarrow\over{#1}}}}\,} 
\newcommand{\liminv}[1]{{\displaystyle{\mathop{\rm lim}_{\buildrel\longleftarrow\over{#1}}}}\,} 
\newcommand{\norm}[1]{\mbox{$\parallel{#1}\parallel$}} 
\newcommand{\boxtensor}{{\Box\kern-9.03pt\raise1.42pt\hbox{$\times$}}} 
\newcommand{\into}{\hookrightarrow} \newcommand{\image}{{\rm image}\,} 
\newcommand{\Lie}{{\rm Lie}\,} 
\newcommand{\CM}{\rm CM}
\newcommand{\sext}{\mbox{${\mathcal E}xt\,$}} 
\newcommand{\shom}{\mbox{${\mathcal H}om\,$}} 
\newcommand{\coker}{{\rm coker}\,} 
\newcommand{\sm}{{\rm sm}} 
\newcommand{\tensor}{\otimes} 
\renewcommand{\iff}{\mbox{ $\Longleftrightarrow$ }} 
\newcommand{\supp}{{\rm supp}\,} 
\newcommand{\ext}[1]{\stackrel{#1}{\wedge}} 
\newcommand{\onto}{\mbox{$\,\>>>\hspace{-.5cm}\to\hspace{.15cm}$}} 
\newcommand{\propsubset} {\mbox{$\textstyle{ 
\subseteq_{\kern-5pt\raise-1pt\hbox{\mbox{\tiny{$/$}}}}}$}} 
\newcommand{\sA}{{\mathcal A}} 
\newcommand{\sB}{{\mathcal B}} \newcommand{\sC}{{\mathcal C}} 
\newcommand{\sD}{{\mathcal D}} \newcommand{\sE}{{\mathcal E}} 
\newcommand{\sF}{{\mathcal F}} \newcommand{\sG}{{\mathcal G}} 
\newcommand{\sH}{{\mathcal H}} \newcommand{\sI}{{\mathcal I}} 
\newcommand{\sJ}{{\mathcal J}} \newcommand{\sK}{{\mathcal K}} 
\newcommand{\sL}{{\mathcal L}} \newcommand{\sM}{{\mathcal M}} 
\newcommand{\sN}{{\mathcal N}} \newcommand{\sO}{{\mathcal O}} 
\newcommand{\sP}{{\mathcal P}} \newcommand{\sQ}{{\mathcal Q}} 
\newcommand{\sR}{{\mathcal R}} \newcommand{\sS}{{\mathcal S}} 
\newcommand{\sT}{{\mathcal T}} \newcommand{\sU}{{\mathcal U}} 
\newcommand{\sV}{{\mathcal V}} \newcommand{\sW}{{\mathcal W}} 
\newcommand{\sX}{{\mathcal X}} \newcommand{\sY}{{\mathcal Y}} 
\newcommand{\sZ}{{\mathcal Z}} \newcommand{\ccL}{\sL} 
 \newcommand{\A}{{\mathbb A}} \newcommand{\B}{{\mathbb 
B}} \newcommand{\C}{{\mathbb C}} \newcommand{\D}{{\mathbb D}} 
\newcommand{\E}{{\mathbb E}} \newcommand{\F}{{\mathbb F}} 
\newcommand{\G}{{\mathbb G}} \newcommand{\HH}{{\mathbb H}} 
\newcommand{\I}{{\mathbb I}} \newcommand{\J}{{\mathbb J}} 
\newcommand{\M}{{\mathbb M}} \newcommand{\N}{{\mathbb N}} 
\renewcommand{\P}{{\mathbb P}} \newcommand{\Q}{{\mathbb Q}} 

\newcommand{\R}{{\mathbb R}} \newcommand{\T}{{\mathbb T}} 
\newcommand{\U}{{\mathbb U}} \newcommand{\V}{{\mathbb V}} 
\newcommand{\W}{{\mathbb W}} \newcommand{\X}{{\mathbb X}} 
\newcommand{\Y}{{\mathbb Y}} \newcommand{\Z}{{\mathbb Z}} 

\title{The lower bound on the HK multiplicities of quadric hypersurfaces}

\author{Vijaylaxmi Trivedi}
\date{}
\address{School of Mathematics, Tata Institute of Fundamental Research, 
Homi Bhabha Road, Mumbai-40005, India }
\email{vija@math.tifr.res.in}

\thanks{}
\begin{abstract}Here we prove that the Hilbert-Kunz mulitiplicity of a quadric hypersurface
of dimension $d$ and odd characteristic $p\geq 2d-4$ is
bounded below by $1+m_d$, where $m_d$ is the $d^{th}$ coefficient in the expansion of 
$\mbox{sec}+\mbox{tan}$. This proves a part of the long standing conjecture 
of Watanabe-Yoshida. We also give an upper bound on the HK multiplicity of 
such a hypersurface.

We approach the question using the HK density function and 
the classification of 
ACM bundles on the smooth quadrics via matrix factorizations.
\end{abstract}
\maketitle
\section{Introduction}

Let $R$ be  a Noetherian ring containing a field of characteristic $p>0$ and
 $I$ be an ideal of finite colength in $R$. For such a pair Monsky ([M]) had introduced a 
characteristic $p$ invariant known as the
Hilbert-Kunz (HK) multiplicity $e_{HK}(R, I)$. 
This is   
a positive real number ($\geq 1$) given by
$$e_{HK}(R, I) = \lim_{n\to \infty}\frac{\ell(R/I^{[q]})}{q^d}.$$

If $(R, {\bf m} , k)$ is a formally unmixed Noetherian local ring  
then it was proved by Watanabe-Yoshida (Theorem~1.5 in [WY1]) that 
$e_{HK}(R,{\bf m}) = 1$ if and only if $R$ is regular.  For the next best class of rings, 
namely  quadric hypersurfaces they made the following  (Conjecture~4.2 in [WY2])

\vspace{5pt}

\noindent{\bf Conjecture}~[WY]\quad{\it Let $p>2$ be  prime and $K = {\bf {\bar {F}_p}}$ and let 
$(R, {\bf m}_R, K)$ be  a formally unmixed nonregular local ring of dimension $n+1$. Then 
$$e_{HK}(R, {\bf m}_R) \geq e_{HK}(R_{p,n+1}, {\bf m}) \geq 1+m_{n+1}.$$}

Here 
$R_{p, n+1} = K[x_0, \ldots, x_{n+1}]/ (x_0^2+\cdots + x_{n+1}^2)$
and  $m_{n+1}$ are the constants occuring as the coefficients of the following expression
$$\mbox{sec}(x)+\mbox{tan}(x) = 1+\sum_{n=0}^{\infty}m_{n+1}x^{n+1},
\quad\mbox{where}\quad |x|<\pi/2.$$

In the same paper ([WY2]) they showed that  the  conjecture holds for $n\leq 3$. 
The second inequality  
of the conjecture for $n\leq 5$ was proved by  Yoshida in [Y]. 
Later the conjecture upto $n\leq 5$ was proved by Aberbach-Enescu in [AE2].

In the context of this
 conjecture, we recall the following result 
(around 2010) of  Gessel-Monsky:
$$\lim_{p\to \infty}e_{HK}(R_{p, n+1}, {\bf m}) = 1+m_{n+1}.$$

\vspace{5pt}

In higher dimensional cases for the class of local formally unmixed nonregular rings of 
fixed dimension $d$, various people ([AE1], [AE2], 
Celikbas-Dao-Huneke-Zhang in  [CDHZ])  have given a 
lower bound $C(d)$ ($> 1$) on the HK multiplicity $e_{HK}(R, {\bf m})$.
However such  lower bounds  $C(d)$ are weaker than the bound given in the above conjecture
as implied by the above result of Gessel-Monsky.

Enescu and Shimomoto in [ES] have  proved the first inequality 
$e_{HK}(R) \geq e_{HK}(R_{p, n+1})$, where $R$ belongs to the class of
 complete intersection local rings.

The conjecture [WY] and related problems have been revisited in the recent paper [JNSWY].

Here  we focus on the second inequality of the above mentioned conjecture 
and prove the following in Section~4.

\begin{thm}\label{t1}Let $p\neq 2$ and  let $p\geq n-2$ for $n$ even and 
let $p\geq 2n-4$ for $n$ odd. Then 
$$1+m_{n+1} + \left(\frac{2n-4}{p}\right) 
\geq e_{HK}(R_{p, n+1}, {\bf m})\geq   1+m_{n+1}.$$
\end{thm}

We approach the invariant by considering 
  the Hilbert-Kunz (HK) density functions for  the pair $(R_{p, n+1}, {\bf m})$.
 where $k$ is a perfect field of 
characteristic $p>0$. The notion 
of HK density function for $(R, I)$, where 
$R$ is a $\N$-graded ring  and $I$ is a homogeneous ideal in $R$ 
of finite colength,  was introduced  by the author ([T]) for standard graded rings and 
later generalized by the author and Watanabe ([TW2]) for 
$\N$-graded rings.  We  recall that the HK density function is a 
compactly supported continuous function $f_{R, I}:[0, \infty]\longto [0, \infty)$ 
defined as
$$f_{R,I}(x) = \lim_{s\to \infty}\ell(R/I^{[q]})_{\lfloor xq \rfloor}, 
\quad\mbox{where}\quad q=p^s$$
and 
$$e_{HK}(R,I) = \int_0^{\infty} f_{R,I}(x)dx.$$

To prove Theorem~\ref{t1} we  prove a stronger result about $\Char~p$ vis-a-vis $\Char~0$ 
(in Section~4):

\begin{thm}\label{t2}The function $f_{R^{\infty}_{n+1}}:[0, \infty)
\longto [0, \infty)$ given by 
$$f_{R^{\infty}_{n+1}}(x) :=\lim_{p\to \infty}f_{R_{p, n+1}, {\bf m}}(x)$$ is a 
well defined continuous function such that $\int_0^\infty f_{R^{\infty}_{n+1}} = 1+m_{n+1}$.

Moreover, if $p\geq 2n-4$ and $n$ is odd, or $p\geq n-2$ and $n$ even then 
$$\begin{array}{lcll}
 f_{R_{p, n+1}, {\bf m}}(x) & = &  f_{R^{\infty}{n+1}}(x) & 
x\in [0,\quad \frac{n+2}{2}-\frac{n-2}{2p}]\\\\
& \geq  &  f_{R^{\infty}_{n+1}}(x) & 
x\in [\frac{n+2}{2}-\frac{n-2}{2p},\quad \frac{n+2}{2}+\frac{n-2}{2p}]\\\\
& = &  f_{R^{\infty}_{n+1}}(x) & 
x\in [\frac{n+2}{2}+\frac{n-2}{2p},\quad \infty).
\end{array}$$
\end{thm}

Note that for $n=1$ and $n=2$ the ring $R_{p, n+1}$ is the homogeneous coordinate ring
of $\P^1_k$ and $\P_k^1\times \P_k^1$ respectively. In both the cases the invariants
  $e_{HK}(R_{p, n+1})$ and 
$f_{R_{p, n+1}, {\bf m}}$ are independent of the characteristic (see 
Eto-Yoshida [EY] and [T]).  Hence we can assume  $n\geq 3$.

Here given $n$ we explicitly write the function $f_{R^{\infty}_{n+1}}$ in Theorem~\ref{f0},
by first writing 
 the function $f_{R_{p, n+1}, {\bf m}}(x)$ for $x\in [0, \infty)\setminus 
[\frac{n+2}{2}-\frac{n-2}{2p},\quad \frac{n+2}{2}+\frac{n-2}{2p}]$.
Hence  we have a computation of $F$-thresholds as a (see Corollary~\ref{c1})

\vspace{5pt} 

\noindent{\bf Corollary}\quad{\it The $F$-thresholds $c^{\bf m}({\bf m}) = n$ 
for $R_{p, n+1}$ defined 
over a perfect field of   characteristic  $p\neq 2$, where 
$p\geq 2n-4$ and $n$ is odd, or $p\geq n-2$ and $n$ even.}

Theorem~\ref{t1} and the result of [ES] prove
the Conjecture~[WY] for the class of complete local rings (for large $p$):

{\it Let $p\neq 2$ and  let $p\geq n-2$ for $n$ even and 
let $p\geq 2n-4$ for $n$ odd.
Let $(R, {\bf m}_R, K)$ be  a formally unmixed nonregular local ring of dimension $n+1$. 
Then  $R$ is a complete intersection ring implies
$$e_{HK}(R, {\bf m}_R) \geq e_{HK}(R_{p,n+1}, {\bf m}) \geq 1+m_{n+1}.$$}

We go about computing the HK density function as follows.
Recall that there exists the complete classification of indecomposable Arithmetically
Cohen-Macaulay (ACM)  bundles (due to 
Buchweitz-Eisenbud-Herzog [BEH]) on smooth quadrics $Q_n = 
\mbox{Proj}~R_{p, n+1}$ in terms of line bundles 
$\sO(t) = \sO_{Q_n}(t)$
and twisted spinor bundles $\sS(t)$ (see Section~2).
Since $F^{s}_*(\sO(a))$ and $F^{s}_*(\sS(a))$ are ACM bundles on $Q_n$, for every
 $s^{th}$ iterated Frobenius map $F^s:Q_n\longto Q_n$ we have
$$F_*^s(\sO(a)) = \oplus _{t\in \Z}\sO(t)^{\nu^s(t, a)} \oplus \oplus_{t\in \Z}
\sS(t)^{\mu^s(t, a)}$$
and 
$$F_*^s(\sS(a)) = \oplus _{t\in \Z}\sO(t)^{{\tilde \nu}^s(t, a)} \oplus \oplus_{t\in \Z}
\sS(t)^{{\tilde \mu}^s(t, a)},$$

Achinger in [A]  showed that   
 the ranks of the bundles $\sO(t)$ and $\sS(t)$ are related 
to the graded components of the ring $R_{p, n+1}/{\bf m}^{[q]}$ by the formula
\begin{equation}\label{*}\ell (R_{p, n+1}/{\bf m}^{[q]})_a = 
\nu^{s}(0, a) + 2\lambda_0 \mu^s(1, a),\end{equation}
where 
${\bf m} = (x_0, \ldots, x_{n+1})$.
This at once  implies that to compute $f_{R_{p, n+1}}$ it is enough to compute all 
{\it the pairs} 
$$\nu^{s}(t, a) + 2\lambda_0 \mu^s(t+1, a),\quad\mbox{for}\quad 
q = p^s>> 0,\quad\mbox{where}~~ t\in \Z,~~~\mbox{and}\quad 0\leq a<q.$$

Now we 
 use another result (Theorem~2 in~[A]) which
 determines, in terms of $q=p^s$, $a$ and $n$,  the occurence of the 
bundle $\sO(t)$ or $\sS(t)$  in  the decomposition 
of $F^s_*(\sO(a))$ and $F^s_*(\sS(a))$.

The layout of the paper is as follows: 

In Section~2 we recall the known results.

In Section~3 we prove that the pairs  
are computable if the decomposition of 
 $F^s_*(\sO(a))$ has only one {\it type} of spinor bundles. However this is not 
always 
the case, 
as the existence of only one type of spinor bundle
would imply that the HK density function
$f_{R_{p, n+1}}$ and thereofore the HK multiplicity $e_{HK}(R_{p, n+1})$ are independent of 
the characteristic $p$. 
However, for large enough $p$ 
one can ensure that there are at the most two types of spinor bundles, as observed in 
 Lemma~\ref{l1}.

We analyse the {\it difficult range} in the interval $[0, 1)$,
with the property that if $a/q$ is outside this range, then the bundle 
$F^s(\sO(a))$ has atmost one type of spinor bundle. In particular 
 every pair $\{\nu^s(t, a)+2\lambda_0\mu^s(t+1, a)\}_t$ is
 computable provided $a/q$ avoids  this range.

Notably   this range keeps shrinking as $p\to \infty$.
We use this observation in Section~4 to
 explicitly write down the  HK density function  everywhere except on the range 
(as in Theorem~\ref{t2}) and also get a 
closed formula for the function $f_{R^\infty_{n+1}}$.

On this range too the HK density function $R_{p, n+1}$ 
can be computed as suggested by the Lemma~\ref{F1} and the computation done  in Section~5
for $n=3$ case.
However the expression will get  more complicated as the case  $n=3$ shows; 
here the function $f_{R_{p, 4}}$ is a
piecewise polynomial and,
on the range $[2+(p-1)/2p, 2+(p+1)/p)$, it is given by  
infinitely many polynomial functions, defined using a nested sequence of intervals.
Looking further, this suggests possible computations for the HK density and related 
invariants in other situations, where we have information on ACM bundles 
using matrix factorizations.

\section{preliminaries}

In this section  we recall the relevant results which are known in the literature.

\begin{defn} A vector bundle $E$ on a smooth $n$-dimensional hypersurface 
$X = \mbox{Proj}~S/(f)$, where $S= k[x_0, \ldots, x_{n+1}]$ 
is called arithmetically Cohen-Macaulay (ACM) if 
$H^i(X, E(m))= 0$, for $0 < i < n$ and for all $m$. 

It is easy to check that 
a vector bundle $E$ on $X$  is ACM if and only if the corresponding graded 
$S/(f)$ module is maximal Cohen-Macaulay (MCM).

\end{defn}

Let $Q_n = \mbox{Proj}~S/(f)$ be the quadric given by the hypersurface 
 $x_0^2+\cdots +x_{n+1}^2 = 0$ in $\P^{n+1}_k = \mbox{Proj}~S$, where $n\geq 3$. 
Let $k$ be an 
algebraically closed field. Henceforth we assume $n>2$.

By B-E-H classification ([BEH]) of indecomposable graded 
MCM modules over
quadrics we have: Other than free modules on $S/(f)$, there is (upto shift) only 
one  indecomposable module $M$ (which is the single spinor bundle $\Sigma$ on $Q_n$)
 if $n$ is odd and there are only two of them 
$M_+$ and $M_-$ (which correspond to the 
two spinor bundles $\Sigma_+$ and $\Sigma_{-}$  on $Q_n$) if $n$ is even.

Morever an MCM module over $S/(f)$ corresponds to  a {\it matrix factorization} of the 
polynomial $f$  (such an equivalence is given by Eisenbud in [E], for more general 
hypersurfaces $(f)$), which is a pair $(\phi, \psi)$ of square matrices of polynomials, 
of the  same size, 
such that 
$\phi\cdot\psi = f\cdot id = \psi\cdot \phi$ and the MCM module is the 
cokernel of $\phi$.

Now  the matrix factorization $(\phi_n, \psi_n)$ for indecomposable bundles on 
$Q_n$ (see Langer [L], Section~2.2)
gives an exact sequence  of locally free sheaves  on $\P_k^{n+1}$
  
\begin{equation}\label{a1}0\longto \sO_{\P_k^N}(-2)^{2^{\lfloor n/2\rfloor +1}}\longby{\Phi_n}
\sO_{\P_k^N}(-1)^{2^{\lfloor n/2\rfloor +1}}\longto i_*\sS\longto 0,\end{equation}
$\sS = \Sigma $ and $\Phi = \phi_n = \psi_n$  for $n$ odd and 
$\sS = \Sigma_{+}\oplus \Sigma_{-}$ and $\Phi_n = \phi_n\oplus \psi_n$  for $n$ even. 
Since $\sS$ is supported on $Q_n$ it is sheaf on $Q_n$.  
Moreover  the above description  gives the short exact sequences of vector bundles on $Q_n$:
If $n$ odd
then 
$$ 0\longto \sS \longto \sO_{Q_n}^{2^{\lfloor n/2\rfloor+1}}\longto \sS(1)\longto 0.$$
If $n$ is even then 
$$ 0\longto \Sigma_{-} \longto \sO_{Q_n}^{2^{\lfloor n/2\rfloor}}\longto \Sigma_{+}(1)
\longto 0$$
and 
$$0\longto \Sigma_{+} \longto \sO_{Q_n}^{2^{\lfloor n/2\rfloor}}\longto \Sigma_{-}(1)
\longto 0.$$

We also have the natural  exact sequence 

\begin{equation}\label{a2}0\longto \sO_{\P_k^N}(-2)\longto 
\sO_{\P_k^N}\longto \sO_{Q_n}\longto 0.\end{equation}

We denote  
$$R_{p, n+1} = \frac{k[x_0, \ldots, x_{n+1}]}{(x_0^2+\cdots+ x_{n+1}^2]} =  
\oplus_{m\geq 0}H^0(Q_n, \sO_{Q_n}(m))\quad\mbox{and}\quad n\geq 3,$$ 
where  $k$ is a field of characteristic $p> 2$.
In particular  the $m^{th}$ graded component of $R_{p, n+1}$ is 
$H^0(Q_n, \sO_{Q_n}(m))$. We will be using the following set of equalities 
in our fothcoming computations.

$$\ell(R_{p, n+1})_m = h^0(Q_n, \sO_{Q_n}(m)) = h^0(\P_k^{n+1}, \sO_{\P_k^{n+1}}(m))-
h^0(\P_k^{n+1}, \sO_{\P_k^{n+1}}(m-2))$$
$$h^0(Q_n, \sS(m)) = 2\lambda_0\left[h^0(\P_k^{n+1}, \sO_{\P_k^{n+1}}(m-1))-
h^0(\P_k^{n+1}, \sO_{\P_k^{n+1}}(m-2))\right],$$

where $2\lambda_0 = 2^{\lfloor n/2\rfloor +1}$. 

By Serre duality ($\omega_{Q_n} = \sO_{Q_n}(-n)$ and $\sS^{\vee} = \sS(1)$)
$$h^n(Q_n, \sO(m)) = h^0(Q_n, \sO(-m-n))\quad\mbox{and}\quad 
 h^n(Q_n, \sS(m)) = h^0(Q_n, \sS(1-m-n)).$$
The rank of $Q_n$-bundle $\sS = \lambda_0$.

Now we recall other relevant facts from [A].

Since $\sO(a)$ and $\sS(a)$ are ACM bundles (also follows from (\ref{a1})), the  projection formula implies that  
$F_*^s(\sO(a))$ is an ACM bundle on $Q_n$. 
For $q=p^s$ and  $a\in \Z$, 

\begin{equation}\label{dc}F_*^s(\sO(a))
 = \oplus _{t\in \Z}\sO(t)^{\nu^s(t, a)} \oplus \oplus_{t\in \Z}
\sS(t)^{\mu^s(t, a)}.\end{equation}

Similarly 
\begin{equation}\label{dcs}F_*^s(\sS(a))
 = \oplus _{t\in \Z}\sO(t)^{{\tilde \nu}^s(t, a)} \oplus \oplus_{t\in \Z}
\sS(t)^{{\tilde \mu}^s(t, a)}.\end{equation}

Then  (see the proof of Theorem~1 of [A])  
considering the short exact sequence
$$0\longto \Omega_{\P^{n+1}_k}(1)\mid_Q\longto 
\oplus^{n+2}\sO_{Q_n} \longto \sO_{Q_n}(1)\longto 0,$$
where the second map is given by $(a_0, \ldots, a_{n+1})\to \sum a_ix_i$, we get 
$$0\longto H^0(Q_n, F^{s*}\Omega_{\P^{n+1}_k}(1)\tensor \sO(a)) 
\longto H^0(Q_n, \oplus F^{s*}\sO(a))\longby{\Psi_{a+q}} H^0(Q_n, \sO(a+q))\longto \cdots$$
This gives
$$\ell(\frac{R_{p, n+1}}{{\bf m}^{[q]}})_{a+q} = \ell(\coker~\Psi_{a+q}) = h^1(Q_n, 
F^{s*}\Omega_{\P^{n+1}_k}(1)\tensor \sO(a)) = h^1(Q_n, 
\Omega_{\P^{n+1}_k}\tensor F^s_*\sO(a+q)).$$
 Now by Lemma~1.2 in [A] we have 
$$h^1(Q_n, \Omega_{\P^{n+1}_k}(t)\mid_{Q_n}) = \delta_{t, 0}\quad and \quad 
h^1(Q_n, \sS\tensor \Omega_{\P^{n+1}_k}(t)\mid_{Q_n}) = 2^{\lfloor n/2\rfloor +1} \delta_{t, 1}.$$ 
Therefore, (replacing $a$ by $a-q$) we  have

\begin{equation}\label{*}\ell(R_{p, n+1}/{\bf m}^{[q]})_a = \coker~{\Psi_{a}}
= \nu^{s}(0, a) + 2\lambda_0 \mu^s(1, a),\end{equation}
where 
${\bf m} = (x_0, \ldots, x_{n+1})$.

We use this observation of [A],  for the computation of 
the HK density function $f_{R_{p, n+1}, {\bf m}}$.
Note that for any integer $m\geq 0$, there is an integer $i\geq 0$ such that 
$iq\leq m <(i+1)q$.
Hence by the  projection formula 
$$F_*^s(\sO(m)) = F^s_*(\sO(m-iq)\tensor F^{s*}(\sO(iq))) = 
F^s_*(\sO(m-iq))\tensor \sO(i).$$
   
 In particular  
\begin{equation}\label{e1}\ell (R_{p, n+1}/{\bf m}^{[q]})_m = 
\nu^s(-i, m-iq) +2\lambda_0 \mu^s({-i+1}, m-iq).\end{equation}

Therefore to know the HK density function 
$f_{R_{p, n+1}, {\bf m}}$, it is enough to 
compute the pair $\nu^s({-i}, a)+2\lambda_0\mu^s({-i+1}, a)$, 
for all $i$ and for $0\leq a<q$.

We also use the following   
 result of Achinger (Theorem~2 in [A]) which  determines, in terms of $s$, $a$ and $n$,
 when the numbers
$\nu^s(i, a)$ and $\mu^s(i, a)$ are nonzero in the decomposition of $F^s_*(\sO(a))$. 
Langer in [L] has given such formula for the occurance of line bundles in 
 the  Frobenius direct image. 

\vspace{5pt}

\noindent{\bf Theorem}\quad[A].\quad  {\it Let $p\neq 2$, $s\geq 1$ and $n\geq 3$ and 
$$F_*^s(\sO(a)) = \oplus _{t\in \Z}\sO(t)^{\nu^s(t, a)} \oplus \oplus_{t\in \Z}
\sS(t)^{\mu^s(t, a)}.$$
Then
\begin{enumerate}
\item $F_*^s(\sO(a))$ contains $\sO(t)$ if and only if $0\leq a-tq \leq n(q-1)$.
\item $F_*^s(\sO(a))$ contains $\sS(t)$ if and only if 
$$\left(\frac{(n-2)(p-1)}{2}\right)\frac{q}{p} \leq a-tq \leq
\left(\frac{(n-2)(p-1)}{2} + n-2 +p\right)\frac{q}{p} -n.$$
\item $F_*^s(\sS(a))$ contains $\sO(t)$ if and only if $1\leq a-tq \leq n(q-1)$.
\item $F_*^s(\sS(a))$ contains $\sS(t)$ if and only if 
$$\left(\frac{(n-2)(p-1)}{2}\right)\frac{q}{p} +1 -\delta_{s,1}\leq a-tq \leq
\left(\frac{(n-2)(p-1)}{2} + n-2 +p\right)\frac{q}{p} -n +\delta_{s,1}.$$
\end{enumerate}}

\section{Formula for  the pairs $\nu^s(i, a)  +2\lambda_0\mu^s({i+1}, a)$}

In the rest of the paper 
$$R_{p, n+1} = \frac{k[x_0, \ldots, x_{n+1}]}{(x_0^2+\cdots+ x_{n+1}^2]}
\quad\mbox{and}\quad Q_n = \mbox{Proj}~R_{p, n+1},$$
where $n\geq 3$ and $k$ is a perfect field of characteristic $p> 2$.

\begin{notations}\label{n1} \begin{enumerate}
\item  $\nu^s(-t, a) = \nu^s_{-t}(a)$ and $\mu^s(-t, a) = \mu^s_{-t}(a)$, where 
$$F_*^s(\sS(a))
 = \oplus _{t\in \Z}\sO(t)^{{\tilde \nu}^s(t, a)} \oplus \oplus_{t\in \Z}
\sS(t)^{{\tilde \mu}^s(t, a)}.$$
\item We also denote $\sO_{Q_n}(m))$ by $\sO(m)$.
\item   $h^0(Q_n, \sO(m)) = Y_m$ and $h^0(\P_k^{n+1}, 
\sO_{\P_k^{n+1}}(m)) = X_m$.
\item Let 
 $$n_0 = \lceil \frac{(n-2)(p-1)}{2p}\rceil\quad\mbox{and}\quad
\Delta = n_0 - \frac{(n-2)(p-1)}{2p}.$$

Now 
$$n \quad\mbox{even}\quad \implies \Delta = \frac{n-2}{2p}\quad\mbox{and}\quad
n \quad\mbox{odd}\quad \implies \Delta = \frac{n-2}{2p} + \frac{1}{2}.$$
\item
A spinor bundle is of {\it type} $t$ if it is isomorphic to $\sS(t)$.
We say two spinor bundles $\sS(t)$ and $\sS(t')$ are of {\it the same type} if $t= t'$.
\item 
An invariant such as 
$\nu^s_i(a)$ is {\it computable} if there exists  a polynomial 
$F_i(X, Y)\in \Q[X, Y]$  of degree $\leq n$ such that $\nu^s_i(a) = F_i(p^s, a)$ 
(similalry for $\mu^s_i(a)$).
\end{enumerate}
\end{notations}

\vspace{5pt}

In  the first lemma we prove that for sufficiently large $p$ (compare to $n$) 
there are atmost  three types of spinor bundles   in the decomposition 
of $F^s_*(a)$, for any $0\leq a <q$.  
Moreover for a fixed such an integer $a$, there are atmost two 
types of spinor bundles.
 
However one can not do better than this, because
if  the decomposition of $F^s_*(\sO(a))$ contains only one type of the spinor bundle then 
all the pairs $\nu^s_i(a)+ 2\lambda_0\mu^s_{i+1}(a)$  are computable
 as will be shown in Lemma~\ref{l2}. But then  the HK density function $f_{R_p, n+1}$ 
and therefore $e_{HK}(R_{p, n+1})$ will be independent of characteristic $p$, which 
 is a  contradiction due to the examples of [WY2].

\begin{lemma}\label{l1}If $0\leq a < q =p^s$ and $p >2$ then 
\begin{enumerate}
\item $F_*^s(\sO(a))$ contains $\sO(t)$ if and only if $t\in \{0, -1, \ldots, -n+1\}$.
Moreover
\item $F_*^s(\sO(a))$ contains $\sS(t)$ implies
 $$t\in \left\{-(n_0-1), -n_0, -(n_0+1), \ldots,   
-\left(n_0+\lceil\frac{n-2}{p}\rceil\right) \right\}.$$

In particular, if $n-2 \leq p$ 
then $n_0 = \lceil n/2\rceil-1$ and
$t \in \{-n_0-1, -n_0, -n_0+1\}$.
Moreover, 

\begin{enumerate}
\item if $n$ is even then 

\begin{enumerate}
\item $\mu^s_{-n_0-1}(a)\neq 0  \implies 0\leq a/q <  \frac{n-2}{2p}$
\item $\mu^s_{-n_0}(a)\neq 0 \implies 0\leq a/q < 1$
\item $\mu^s_{-n_0+1}(a)\neq 0 \implies 1-\frac{n-2}{2p} \leq a/q$.
\end{enumerate}

\item If $n$ is odd then 

\begin{enumerate}
\item $\mu^s_{-n_0-1}(a) = 0$
\item $\mu^s_{-n_0}(a)\neq 0 \implies 0\leq a/q <  \frac{1}{2}+\frac{n-2}{2p}$
\item $\mu^s_{-n_0+1}(a)\neq 0 \implies \frac{1}{2} - \frac{n-2}{2p} \leq a/q $.
\end{enumerate}

\end{enumerate}
\end{enumerate}

\end{lemma}
\begin{proof}The assertion~(1) is just restating the assertion (1) of [A].

By the  assertion~(2) of [A], if $\sS(t)$ occurs in $F^s_*(\sO(a))$ then 
$$ (n_0- \Delta)q \leq a -tq \leq  (n_0- \Delta)q + (n-2)q/p + q -n.$$
$$ (n_0- \Delta) \leq a/q -t \leq  (n_0- \Delta) + (n-2)/p + 1 -n/q.$$

Hence $$0\leq \frac{a}{q} + \Delta -t-n_0 \leq \frac{n-2}{p} + 1 -\frac{n}{q}.$$
 Now $n_0-1 \leq -t$ as $a/q+\Delta < 2$.
On the other hand
$$-t-n_0 -1 \leq \frac{n-2}{p}-\frac{n}{q}\implies -t-n_0-1 < \frac{n-2}{p}
\leq \lceil\frac{n-2}{p}\rceil $$
This implies $-t\leq n_0+\lceil\frac{n-2}{p}\rceil.$

This proves  assertion~(2):
$n_0-1
\leq -t \leq \lceil {(n-2)}/{p}\rceil + n_0.$

In particular   $n-2 \leq p$ implies  $-t\in \{n_0-1, n_0, n_0+1\}$. 

 Now the  rest  of the assertion follows
from the following three possibilities:
\begin{enumerate}
\item If $-t = n_0-1$ then we 
have $0\leq a/q +\Delta -1 \leq 1 + \frac{n-2}{p}-\frac{n}{q} < 1 + \frac{n-2}{p}$.
\item If $-t = n_0$ then we 
have $0\leq a/q +\Delta  \leq 1 + \frac{n-2}{p}-\frac{n}{q} < 1 + \frac{n-2}{p} $.
\item If $-t = n_0+1$ then we 
have $0\leq a/q +\Delta  \leq  \frac{n-2}{p}-\frac{n}{q} < \frac{n-2}{p}$.
\end{enumerate}

\end{proof}

\begin{rmk}If $n$ is even and $p\geq n-2$ then  
$F^s_*(\sO(a))$ as atmost three  types of spinor bundles .
If $0\leq a<q$, then they all belong to the set 
$\{\sS(-n_0+1), \sS(-n_0), \sS(-n_0-1)\}$ (we will see that  the  only first two are 
relevant for our computation). 
Moreover for a given choice of such an integer $a$, there are atmost two 
types of spinor bundles, namely 
$\sS(-n_0+1)$, $\sS(-n_0)$ or $\sS(-n_0)$, $\sS(-n_0-1)$.

If $n$ is odd  and $p\geq n-2$ and $0\leq a <q$  then 
the only possible   spinor bundles are  
$\sS(-n_0+1)$, $\sS(-n_0)$.
\end{rmk}

\begin{notations}\label{n2} Let $0\leq a < q= p^s$ be an  integer.
We define iteratively, $Z^s_{-i}(a)$ for $0\leq i \leq n_0+2$
 as follows.
$$\mbox{Let}\quad Z^s_0(a) = Y_a,\quad Z^s_{-1}(a) = Y_{a+q}-Y_1Y_a$$
 and let 
$$Z^s_{-i}(a) = Y_{a+iq} - \left[Y_1Z^s_{-i+1}(a)+Y_2Z^s_{-i+2}(a)+\cdots
+Y_iZ^s_0(a)\right].$$
 
Similarly 
we define iteratively, $L^s_{-i}(a)$ for $n_0+1\leq i \leq n-1$ 
 as follows.

Let $L^s_{-n+1}(a) = Y_{q-a-n}$ and
for $n_0+1\leq i < n-1$ we denote 
 $$ L^s_{-i}(a) = Y_{(n-i)q-a-n}-\left[Y_{n-i-1}L^s_{-n+1}(a)+
\cdots + Y_1L^s_{-i-1}(a)\right].$$ 
\end{notations}

\begin{rmk}\label{r3} By construction 
 $$Z^s_{-i}(a) = r_{i0}Y_a + r_{i1}Y_{a+q}+ \cdots  + r_{i(i-1)}Y_{a+(i-1)q}+Y_{a+iq}$$
and $$L^s_{-i}(a) = s_{i0}Y_{q-a-n} + s_{i1}Y_{2q-a-n}+ \cdots  + 
s_{i(n-i-2)}Y_{(n-i-1)q-a-n}+Y_{(n-i)q-a-n},$$
where $\{r_{ij}, s_{ik}\}_{j, k}$ are some rational numbers independent of $s$ and $a$.
On the other hand for an integer $m\geq 0$ 
$$Y_m = \binom{m+n+1}{n+1}-\binom{m+n-1}{n+1} = \frac{2m^n}{n!}+O(m^{n-1}).$$

Hence both $Z_{-i}^s(a)$ and $L_{-i}^s(a)$  are computable in the sense of 
Notations~\ref{n1}.\end{rmk}

The following lemma 
implies that except for $i = n_0+1$ and $i = n_0+2$ 
all the pairs  $\nu^s_{-i}(a)+2\lambda_0\mu^s_{-i+1}(a)$ are computable.

\begin{lemma}\label{l2} If $p\geq n-2$ is an odd prime and $n\geq 3$. Then for given 
$1\leq s$ and $0\leq a <q = p^s$  
\begin{enumerate}

\item $\nu^s_{-i}(a)+2\lambda_0\mu^s_{-i+1}(a) = Z^s_{-i}(a)$, if $0\leq i \leq n_0$.\\\

\item $\nu^s_{-n_0-1}(a)+2\lambda_0\mu^s_{-n_0}(a) = 
{{Z^s}}_{-n_0-1}(a) + 2\lambda_0\mu^s_{-n_0+1}(a)$.\\\

\item $\nu^s_{-n_0-2}(a)+2\lambda_0\mu^s_{-n_0-1}(a) = 
{{Z^s}}_{-n_0-2}(a) + 2\lambda_0(Y_1-Y_0)\mu^s_{-n_0+1}(a) +
 2\lambda_0\mu^s_{-n_0}(a)$.\\\

\item $\nu^s_{-i}(a)+2\lambda_0\mu^s_{-i+1}(a) = L^s_{-i}(a)$, if  $n_0+3\leq i\leq n-1$.\\\

\item $\nu^s_{-n_0-2}(a) = L^s_{-n_0-2}(a)$.\\\

\item $\nu^s_{-n_0-1}(a) = L^s_{-n_0-1}(a) - 2\lambda_0\mu^s_{-n_0-1}(a)$.\\\

\item $\nu^s_{-i}(a) = 0$, for $i\geq n$ and $\mu^s_{-j}(a) =0$ if 
$j\not\in \{n_0+1, n_0, n_0-1\}$.

\end{enumerate}
\end{lemma}
\begin{proof}
We fix $0\leq a<q = p^s$. Then, by Lemma~\ref{l1}
$$\begin{array}{lcl}
F^s_*(\sO(a)) & = & \sO(-n+1)^{\nu^s_{-n+1}(a)}\oplus\cdots\oplus
\sO(-1)^{\nu^s_{-1}(a)}\oplus \sO^{\nu^s_0(a)}\\\
& & \oplus \sS(-n_0+1)^{\mu^s_{-n_0+1}(a)}\oplus 
\sS(-n_0)^{\mu^s_{-n_0}(a)}\oplus \sS(-n_0-1)^{\mu^s_{-n_0-1}(a)}.\end{array}$$

Tensoring the above equation by $\sO(i)$ and by projection formula, we get

$$\begin{array}{lcl}
F^s_*(\sO(a+iq)) & = & \sO(i-n+1)^{\nu^s_{-n+1}(a)}\oplus\cdots\oplus
\sO(i-1)^{\nu^s_{-1}(a)}\oplus \sO(i)^{\nu^s_{0}(a)}\\\
& & \oplus \sS(i-n_0+1)^{\mu^s_{-n_0+1}(a)}\oplus
\sS(i-n_0)^{\mu^s_{-n_0}(a)}\oplus 
\sS(i-n_0-1)^{\mu^s_{-n_0-1}(a)}.\end{array}$$

Applying the the functor $H^0(Q_n, -)$ we get
\begin{equation}\label{**}
\begin{split}
 \nu^s_0(a)  = &  Y_a = Z^s_0(a)\\
\nu^s_{-1}(a) = &   Y_{a+q} - Y_1Y_a = Z^s_{-1}(a)\\
\end{split}
\end{equation}
In general, for $i<n_0$, 
$$Y_{a+iq} = Y_0\nu^s_{-i}(a) + Y_1\nu^s_{-i+1}(a) + \cdots + Y_i\nu^s_0(a)
$$
which implies
$$\nu^s_{-i}(a) =  \nu^s_{-i}(a)+2\lambda_0\mu^s_{-i+1}(a)  = Z^s_{-i}(a). $$

For $i= n_0$
$$Y_{a+n_0q} = Y_0\nu^s_{-n_0}(a) + Y_1\nu^s_{-n_0+1}(a) + \cdots + Y_{n_0}\nu^s_0(a)
+ 2\lambda_0\mu^s_{-n_0+1}(a)$$
$$ = \nu^s_{-n_0}(a) + Y_1Z^s_{-n_0+1}(a) + \cdots + Y_{n_0}Z^s_0(a)
+ 2\lambda_0\mu^s_{-n_0+1}(a)$$

$$\implies 
\nu^s_{-n_0}(a)+2\lambda_0\mu^s_{-n_0+1}(a) = Z^s_{-n_0}(a)$$
This proves assertion~(1).

For $i= n_0+1$
$$Y_{a+(n_0+1)q} = Y_0\nu^s_{-n_0-1}(a) + Y_1\nu^s_{-n_0}(a) + \cdots + Y_{n_0+1}\nu^s_0(a)
+2\lambda_0\left[(Y_1-Y_0)\mu^s_{-n_0+1}(a)+ \mu^s_{-n_0}(a)\right].$$ 

$$\implies 
\nu^s_{-n_0-1}(a)+2\lambda_0\mu^s_{-n_0}(a) = Z^s_{-n_0-1}(a) + 
2\lambda_0\mu^s_{-n_0+1} $$
This proves assertion~(2).

For $i= n_0+2$
$$\begin{array}{lcl}
Y_{a+(n_0+2)q} & = &  Y_0\nu^s_{-n_0-2}(a) + Y_1\nu^s_{-n_0-1}(a) + \cdots + 
Y_{n_0+2}\nu^s_0(a)\\\\
& & +2\lambda_0(X_2-X_1)\mu^s_{-n_0+1}(a)+2\lambda_0(X_1-X_0)\mu^s_{-n_0}(a)+
2\lambda_0\mu^s_{-n_0-1}(a).\end{array}$$ 
This implies
$$ \nu^s_{-n_0-2}(a)+2\lambda_0\mu^s_{-n_0-1}(a) = Z^s_{-n_0-2}(a) + 
2\lambda_0(Y_1-Y_0)\mu^s_{-n_0+1}(a)+2\lambda_0\mu_{-n_0}(a).$$
This proves assertion~(3).

Now we tensor the decomposition of $F^s_*(\sO(a))$ by $\sO(-j)$
 and apply the functor $H^n(Q_n, -)$.
 By  duality 
$h^n(Q_n, \sO_{Q_n}(m)) = h^0(Q_n, \sO_{Q_n}(-m-n)) = Y_{-m-n}$, hence  

$$Y_{jq-a-n} = Y_{j-1}\nu_{-n+1}^s(a)+ Y_{j-2}\nu_{-n+2}^s(a)+\cdots +
Y_{0}\nu_{-n+j}^s(a)+\mu_{-n_0+1}h^0(Q_n, S(n_0+j-n))$$
$$+\mu_{-n_0}h^0(Q_n, S(n_0+j+1-n))+
\mu_{-n_0-1}h^0(Q_n, S(n_0+j+2-n)).$$
Hence   $1\leq j \leq n-n_0-2 $, we get 
$$Y_{jq-a-n} = Y_{j-1}\nu_{-n+1}^s(a)+ Y_{j-2}\nu_{-n+2}^s(a)+\cdots +
Y_0\nu_{-n+j}^s(a).$$

Hence 
$$\nu_{-n+j} = Y_{jq-a-n} -\left[Y_{j-1}\nu^s_{-n+1}(a) + \cdots +
Y_2\nu^s_{-n+j-2}(a)+Y_1\nu^s_{-n+j-1}(a)\right] = L_{-n+j}.$$

\begin{equation}\label{*l}
\begin{split}
\mbox{For}\quad  j = 1,\quad \nu^s_{-n+1}(a)  = & Y_{q-a-n},\\
\mbox{For}\quad  j = 2,\quad \nu^s_{-n+2}(a)  = & Y_{2q-a-n} - Y_1\nu^s_{-n+1}(a)
\end{split}\end{equation}
In general

$$\begin{array}{lcl}
\nu^s_{-i}(a)  & = & \nu^s_{-i}(a)+2\lambda_0\mu^s_{-i+1}(a) =  L^s_{-i}(a),\quad{for}\quad  
n_0+3 \leq i \leq n-1\\\
\nu^s_{-n_0-2}(a) & = &  L^s_{-n_0-2}(a).\end{array}$$
This proves assertions~(4) and (5).

\vspace{5pt}

For $j = n-n_0-1$ we get
$$Y_{(n-n_0-1)q-a-n} = Y_{n-n_0-2}\nu^s_{-n+1}(a)+ \cdots  + Y_0\nu^s_{-n_0-1}(a) + 
2\lambda_0\mu^s_{-n_0-1}(a).$$

\begin{equation}\label{n_0+1}
\nu^s_{-n_0-1}(a) = L^s_{-n_0-1}(a) - 2\lambda_0\mu^s_{-n_0-1}(a).\end{equation}
This proves assertion~(6) and hence the lemma.
\end{proof}

\begin{rmk}\label{r2} By the above set of eualities it follows that, for a given $a$,
if  there is at the most
one type of spinor bundle in the decomposition of $F^s_*(\sO(a))$ then 
all the pairs $\nu^s_{-i}(a)+2\lambda_0\mu^s_{-i+1}(a)$ are computable.

In the next lemma we use this observation to 
classify  the $(a, q)$  for which all the pairs are computable. It is  enough to check 
this for the pairs
$\nu^s_{-n_0-1}(a)+2\lambda_0\mu^s_{-n_0}(a)$ and 
$\nu^s_{-n_0-2}(a)+2\lambda_0\mu^s_{-n_0-1}(a)$, as, by Lemma~\ref{l2},
 rest of the  other pairs are computable.
\end{rmk}

\begin{lemma}\label{l3}Let $0\leq a < q=p^s$, where $p> 2$ and $n\geq 3$. 
\begin{enumerate}
\item If $n$ is even and $p\geq n-2$ then

\begin{enumerate}
\item $\nu^s_{-n_0-1}(a)+2\lambda_0\mu^s_{-n_0}(a)
 = Z^s_{-n_0-1}(a),\quad\mbox{if}\quad
 \mbox{$0\leq \frac{a}{q} < 1-\frac{(n-2)}{2p}$}.$

\item $\nu^s_{-n_0-2}(a)+2\lambda_0\mu^s_{-n_0-1}(a)$   $$=  \begin{cases}
{{Z^s}}_{-n_0-2}(a)  +  2\lambda_0\mu^s_{-n_0}(a) & \mbox{if}\quad
0\leq \frac{a}{q} < \frac{(n-2)}{2p}.\\  
\nu^s_{-n_0-2}(a) = L^s_{-n_0-2}(a) &\mbox{if}\quad \frac{(n-2)}{2p} 
\leq \frac{a}{q}.\end{cases}$$
\end{enumerate}

\noindent{In} particular all the pairs are 
computable for
 $a/q \in [0, 1)\setminus [0, \frac{n-2}{p})\cup [1-\frac{n-2}{p}, 1)$.

\vspace{5pt}
 
\item If  $n$ is odd and  $p\geq 2n-4$  
then 
\begin{enumerate}
\item  $\nu^s_{-n_0-1}(a)+2\lambda_0\mu^s_{-n_0}(a)$  
$$ = \begin{cases} Z^s_{-n_0-1}(a) &
\mbox{if}~~ 0\leq \frac{a}{q} < \frac{1}{2}-\frac{(n-2)}{2p}\\
  \nu^s_{-n_0-1}(a) = 
L^s_{-n_0-1}(a) &~~\mbox{if}~~ \frac{1}{2}+\frac{(n-2)}{2p} \leq \frac{a}{q} < 1.\end{cases}$$   
\item $\nu^s_{-n_0-2}(a)+2\lambda_0\mu^s_{-n_0-1}(a)   = 
\nu^s_{-n_0-2}(a) =  L^s_{-n_0-2}(a).$
\end{enumerate}

\vspace{5pt}

\noindent{In} particular all the pairs are 
computable for
 $a/q\in [0, 1)\setminus [\frac{1}{2}-\frac{n-2}{2p}, \frac{1}{2}+\frac{n-2}{2p})$
 \end{enumerate}
\end{lemma}
\begin{proof}Let $\delta =1$ if $n$ is even and $\delta = 1/2$ if $n$ is odd.

\vspace{5pt}

\begin{enumerate}
\item {If} $0\leq a/q <\delta -(n-2)/2p$ then $\mu^s_{-n_0+1}(a)=0$.
Hence the assertion 
$$\nu^s_{-n_0-1}(a)+2\lambda_0\mu^s_{-n_0}(a) = Z^s_{-n_0-1}(a)$$ 
follows from Lemma~\ref{l2}~(2).

\item {If} $\delta + (n-2)/2p \leq a/q <1$ (which holds only if $n$ is odd) then $\mu^s_{-n_0-1}(a) = 0$ and $\mu^s_{-n_0}(a) = 0$ 
hence by Lemma~\ref{l2}~(6)
$$\nu^s_{-n_0-1}(a)+2\lambda_0\mu^s_{-n_0}(a) = 
\nu^s_{-n_0-1}(a) = L^s_{-n_0-1}(a).$$

\item{If} $0\leq a/q < (n-2)/2p$ then, by Lemma~\ref{l1}~(2), 
  $\mu^s_{-n_0+1}(a) = 0$ and hence 
the equality 
$$\nu^s_{-n_0-2}(a)+2\lambda_0\mu^s_{-n_0-1}(a) = 
{{Z^s}}_{-n_0-2}(a)  +  2\lambda_0\mu^s_{-n_0}(a)$$
follows from Lemma~\ref{l2}~(3).

\item{If} $(n-2)/2p \leq a/q <1$ or if $n$ is odd then $\mu^s_{-n_0-1}(a) = 0$. Hence 
by Lemma~\ref{l2}~(5), we have the equality 
$\nu^s_{-n_0-2}(a)+2\lambda_0\mu^s_{-n_0-1}(a)   =  L^s_{-n_0-2}(a)$.

\end{enumerate}
\end{proof}

\vspace{5pt}

Proposition~5.2 of [L] states that 
if $s=1$ then $F_*(\sO(a))$ and $F_*\sS(a))$ both have  atmost one type of spinor bundle, 
Here, using explicit formulation of [A] we write them down  explicitly.
In particular all the pairs and $\nu^1_i(a)$, $\mu^1_i(a)$, ${\tilde \nu}^1_i(a)$, 
and ${\tilde \mu}^1_i(a)$  are computable, where ${\tilde \nu}^1_i(a) = 
{\tilde \nu}^1(i, a)$ and ${\tilde \mu}^1_i(a) = 
{\tilde \mu}^1(i, a)$
are the integers as in (\ref{dcs}).

\begin{lemma}\label{F1}If $p\neq 2$ and $n=3$ then 
for $0\leq a <p$, we have 
$$F_*(\sO(a))  =  \begin{cases}
\sO(-2)^{\nu^1_{-2}(a)}\oplus
\sO(-1)^{\nu^1_{-1}(a)}\oplus \sO^{\nu^1_0(a)}
\oplus \sS(-1)^{\mu^s_{-1}(a)},& \mbox{if}\quad a\leq \frac{p-1}{2}-{2}\\\
 
\sO(-2)^{\nu^1_{-2}(a)}\oplus
\sO(-1)^{\nu^1_{-1}(a)}\oplus \sO^{\nu^1_0(a)},&  \mbox{if}\quad a = \frac{p-1}{2}-1\\\

\sO(-2)^{\nu^1_{-2}(a)}\oplus
\sO(-1)^{\nu^1_{-1}(a)}\oplus \sO^{\nu^1_0(a)}
\oplus \sS(0)^{\mu^s_{0}(a)},& \mbox{if}\quad a\geq \frac{p-1}{2}.\end{cases}$$

Moreover $4\mu_{-1}^1(a) = Y_{a+2p}-Y_1Y_{a+p}+(Y_1^2-Y_2)Y_a -Y_{p-a-3}$, if 
$a\leq \frac{p-1}{2}-{2}$.

Also 

$$F_*(\sS(a))  =  \begin{cases}
\sO(-2)^{\tilde{\nu}^1_{-2}(a)}\oplus
\sO(-1)^{\tilde{\nu}^1_{-1}(a)}\oplus \sO^{\tilde{\nu}^1_0(a)}
\oplus \sS(-1)^{\tilde{\mu}^s_{-1}(a)},& \mbox{if}\quad a\leq \frac{p-1}{2}-{1}\\\
\sO(-2)^{\tilde{\nu}^1_{-2}(a)}\oplus
\sO(-1)^{\tilde{\nu}^1_{-1}(a)}\oplus \sO^{\tilde{\nu}^1_0(a)}
\oplus \sS(0)^{\tilde{\mu}^s_{0}(a)},& \mbox{if}\quad a\geq \frac{p-1}{2}.\end{cases}$$

Moreover, if 
$a\leq \frac{p-1}{2}-{1}$ then 
$$4{\tilde {\mu}}_{-1}^1(a) = h^0(Q_3, \sS(a+2p))- Y_1h^0(Q_3, \sS(a+p))
+(Y_1^2-Y_2)h^0(Q_3, \sS(a))- h^0(Q_3, \sS(p-a-2)).$$

In general, for any given  $n\geq 3$ and $0\leq a<p$ the bundle
 $F_*(\sO(a))$ (similarly  $F_*(\sS(a))$) can not contain 
both $\sS(t)$ and $\sS(t')$, where $t\neq t'$.
\end{lemma}

\begin{proof}We first prove the last assertion for $n\geq 3$. By Theorem~2 of [A], $F_*\sO(a))$ contains $\sS(t)$ if and only if 
$$\frac{(n-2)(p-1)}{2}\leq a-tp \leq \frac{(n-2)(p-1)}{2} + (p-2).$$

Since the difference between the maximum and minimum is $\leq p-1$,
 there can not be two different $t$ and $t'$ satisfying 
such equation.
Similar assertion holds for $F_*(\sS(a))$
$F_*(\sS(a))$ contains $\sS(t)$ if and only if 
$$\frac{(n-2)(p-1)}{2}\leq a-tp \leq \frac{(n-2)(p-1)}{2} + (p-1).$$

It is easy to work out  $n=3$ case. The formula for ${\tilde \mu}^1_{-1}(a)$ can 
be worked out as follows: 

Let $a < (p-1)/2$. Tensoring the equation 
$$F_*(\sS(a))  =  
\sO(-2)^{\tilde{\nu}^1_{-2}(a)}\oplus
\sO(-1)^{\tilde{\nu}^1_{-1}(a)}\oplus \sO^{\tilde{\nu}^1_0(a)}
\oplus \sS(-1)^{\tilde{\mu}^s_{-1}(a)}$$
by $\sO(i)$ we get 
$$F_*(\sS(a+ip))  =  
\sO(i-2)^{\tilde{\nu}^1_{-2}(a)}\oplus
\sO(i-1)^{\tilde{\nu}^1_{-1}(a)}\oplus \sO(i)^{\tilde{\nu}^1_0(a)}
\oplus \sS(i-1)^{\tilde{\mu}^s_{-1}(a)}.$$
Applying the functor $h^0(Q_n, -)$ for $i=0, 1$ and $2$
Now we have ${\tilde{\nu}^1_0(a)} = h^0(Q, \sS(a))$, ${\tilde{\nu}^1_{-1}(a)}=
h^0(Q, \sS(a+p))- Y_1{\tilde{\nu}^1_0(a)}$ and 
$$4{\tilde{\mu}^1_{-1}(a)} = h^0(Q, \sS(a+2p))- \left[Y_2{\tilde{\nu}^1_0(a)}
+Y_1{\tilde{\nu}^1_{-1}(a)} + {\tilde{\nu}^1_{-2}(a)}\right].$$

\end{proof}

\section{The HK density function $f_{R_{p, n+1}, {\bf m}}$ and $f_{R^{\infty}_{n+1}, 
{\bf m}}$}

\begin{rmk}\label{n3}
Let $Z_{-i}^s(a)$ and $L_{-i}^s(a)$ be the numbers as Notations~\ref{n2}.
Then we can write 
$$Z^s_{-i}(a) = \sum_{j=0}^ir_{ij}Y_{a+jq}
  \quad\mbox{and}\quad 
L^s_{-i}(a) = \sum_{j =0}^{n-i-1}s_{ij}Y_{(j+1)q-a-n}$$
where $\{r_{ij}, s_{ik}\}_{j, k}$
are rational numbers independent  of $a$ and $s$,
Now if $x\geq 0$ such that $xq_0\in \Z_{\geq 0}$ for some $q_0 = p^{s_0}$ and  
if $i$ is the integer such that $0\leq xq_0-iq_0 < q_0$ then
$\lim_{q\to \infty} (xq-iq)/q = x-i$.
 This observation implies that, if we define the functions ${\bf Z_{-i}}$ and 
${\bf L_{-i}}$ on the interval $[i, i+1)$ by 
$${\bf Z_{-i}}(x) := \lim_{q\to \infty}\frac{Z^s_{-i}(\lfloor xq\rfloor-iq)}{q^n} 
\quad\mbox{and}\quad
{\bf L_{-i}}(x) := \lim_{q\to \infty}\frac{L^s_{-i}(\lfloor xq\rfloor-iq)}{q^n}.$$ 
 then we have
$${\bf Z_{-i}}(x) = 
\frac{2}{n!}\left[r_{i0}(x-i)^n+r_{i1}(x-i+1)^n+\cdots +r_{ii}(x)^n\right]
$$
and 
$${\bf L_{-i}}(x) = 
\frac{2}{n!}\left[s_{i0}(i+1-x)^n+s_{i1}(i+2-x)^n+\cdots +s_{i(n-i-1)}(n-x)^n\right].
$$
\end{rmk}

\begin{lemma}\label{l4}\begin{enumerate}
\item If $n\geq 4$ is  an even number and $p\geq n-2$ and $p\neq 2$ 
Then 

$$f_{R_{p, n+1}}(x) =$$
 $$ \begin{cases}  {\bf Z_{-i}}(x), \quad\mbox{if}\quad 
 i\leq x < i+1 \quad\mbox{and}\quad 0\leq i \leq  n_0\\\\
{\bf  Z}_{-n_0-1}(x),\quad\mbox{if}\quad  
(n_0+1) \leq x < (n_0+2) -\frac{n-2}{2p}\\\\
 {\bf  Z}_{-n_0-1}(x)+2\lambda_0\displaystyle{\lim_{q\to \infty}
\frac{\mu^s_{-n_0+1}(\lfloor xq\rfloor -(n_0+1)q)}{q^n}},
\quad\mbox{if\quad  $1-\frac{n-2}{2p}\leq x -(n_0+1)< 1$} \\\\
{\bf  Z}_{-n_0-2}(x)+ 2\lambda_0\displaystyle{\lim_{q\to \infty}
\frac{\mu^s_{-n_0}((\lfloor xq\rfloor -(n_0+2)q)}{q^n}},
\quad\mbox{if\quad  $0\leq x -(n_0+2) < \small{\frac{n-2}{2p}}$}\\\\
 {\bf L}_{-n_0-2}(x),\quad\mbox{if}\quad  
(n_0+2)+\frac{n-2}{2p} \leq x < (n_0+3)\\\\
  {\bf L}_{-i}(x),\quad\mbox{if}\quad 
 i\leq x < i+1 \quad\mbox{and}\quad n_0+3\leq i < n\end{cases}$$
\vspace{5pt}

and $f_{R_p, n+1}(x) = 0$ otherwise.

\vspace{10pt}

\item If $n\geq 3$ is  an odd number and $2n-4 \leq p$ and $p\neq 2$ then 

$$f_{R_p, n+1}(x)  = $$
$$\begin{cases} {\bf Z_{-i}}(x) 
 \quad\mbox{if}\quad 
 i\leq x < i+1 \quad\mbox{and}\quad 0\leq i \leq  n_0\\\\
 {\bf  Z}_{-n_0-1}(x), \quad\mbox{if}\quad  
(n_0+1) \leq x < (n_0+\frac{3}{2}) -\frac{n-2}{2p}\\\\
{\bf  Z}_{-n_0-1}(x)+2\lambda_0\displaystyle{\lim_{q\to \infty}
\frac{\mu^s_{-n_0+1}(\lfloor xq\rfloor -(n_0+1)q)}{q^n}},\quad
\mbox{if\quad ${ \frac{1}{2}-\frac{n-2}{2p}\leq x -(n_0+1) < 
\frac{1}{2}+\frac{n-2}{2p}}$}\\\\
{\bf L}_{-n_0-1}(x),\quad\mbox{if}\quad  
(n_0+1)+\frac{1}{2}+\frac{n-2}{2p} \leq x < (n_0+2)\\\\
{\bf L}_{-i}(x),\quad\mbox{if}\quad 
 i \leq x < i+1 \quad\mbox{and}\quad n_0+2\leq i <  n
\end{cases}$$
\vspace{5pt}

and $f_{R_p, n+1}(x) = 0$ otherwise.
\end{enumerate}
\end{lemma}
\begin{proof}Let $q=p^s$ and $m\in \Z$ and
let $\nu^s_t(m)$ and $\mu^s_t(m)$ be the 
numbers occuring in   the decomposition 
$$F_*^s(\sO(m)) = \oplus _{t\in \Z}\sO(t)^{\nu^s_t(m)} \oplus \oplus_{t\in \Z}
\sS(t)^{\mu^s_t(m)}.$$

If  $m\geq 0$ is integer then there is  $i\geq 0$  an integer such that 
$0\leq m-iq < q$. Now, by (\ref{e1}), 
$$\ell (R_{p, n+1}/{\bf m}^{[q]})_{m} = \nu^s_0({a+iq})+
2\lambda_0\mu^s_1({a+iq}) =  
\nu^s_{-i}(a) +2\lambda_0 \mu^s_{-i+1}(a).$$

Here we write the details when $n$ is even, the case when $n$ is odd 
follows along the same lines.

Lemma~\ref{l2}~(1) gives
$$\ell (\frac{R_{p, n+1}}{{\bf m}^{[q]}})_{a+iq}  =  Z^s_{-i}(a)\quad\mbox{for 
every}\quad 0\leq i\leq n_0\quad\mbox{and for}\quad 0\leq a<q$$
By Lemma~\ref{l3} and Lemma~\ref{l2}~(2), we have ,
$$\begin{array}{lcl}
\ell (\frac{R_{p, n+1}}{{\bf m}^{[q]}})_{a+(n_0+1)q} & = & \begin{cases}
  Z^s_{-n_0-1}(a)&\quad\mbox{if}\quad  
0\leq a < q(1 -\frac{n-2}{2p})\\\\
 Z^s_{-n_0-1}(a)+
2\lambda_0\mu^s_{-n_0+1}(a) & \quad\mbox{if}\quad  
q-\frac{n-2}{2p}q\leq a < q \end{cases}\\\\
\ell (\frac{R_{p, n+1}}{{\bf m}^{[q]}})_{a+(n_0+2)q} & = & \begin{cases}
 Z^s_{-n_0-2}(a)+
2\lambda_0\mu^s_{-n_0}(a)&\quad\mbox{if}\quad  
0 \leq a <  \frac{n-2}{2p}q\\\\ 
L^s_{-n_0-2}(a) & \quad\mbox{if}\quad  
\frac{n-2}{2p}q\leq a < q \end{cases}
\end{array}$$
By Lemma~\ref{l2}~(4)
$$\ell(\frac{R_{p, n+1}}{{\bf m}^{[q]}})_{a+jq}  =  
L^s_{-j}(a)\quad\mbox{for every}\quad n_0+3\leq j\leq n-1\quad\mbox{and for}\quad 0\leq a<q.$$
and $\ell(\frac{R_{p, n+1}}{{\bf m}^{[q]}})_{m} = 0$ otherwise.

By definition
 $$f_{R_{p, n+1}}(x) =  \lim_{s\to \infty}\frac{1}{q^n}
\ell(R_{p, n+1}/{\bf m}^{[q]})_{\lfloor xq\rfloor}$$
and is a continuous function and the set 
$\{x\in \R\mid xq\in \Z,\quad\mbox{for some}\quad q=p^s\}$ is a dense of $\R$.
Hence the theorem follows from Remark~\ref{n3}.
\end{proof}

\begin{thm}\label{f0} The function 
  $f^{\infty}_{R_{n+1}}:[0, \infty)\longto [0,\infty)$ 
given by 
$$f^{\infty}_{R_{n+1}}(x) := \lim_{p\to \infty}f_{R_{p, n+1}}(x)$$
is partially symmetric continuous function, that is 
$$f^{\infty}_{R_{n+1}}(x) =  f^{\infty}_{R_{n+1}}(n-x),\quad\mbox{for}\quad 
0\leq x\leq (n-2)/2$$
and 
is described as follows:

\begin{enumerate}
\item If $n\geq 4$ is even then 
  $$f^{\infty}_{R_{n+1}}(x)  = \begin{cases}  {\bf Z_{-i}}(x)& \quad\mbox{if}\quad 
 i\leq x < i+1 \quad\mbox{and}\quad 0\leq i \leq  n_0+1\\\\
 {\bf L}_{-n_0-2}(x) & \quad\mbox{if}\quad  
(n_0+2)\leq x < (n_0+3)\\\\
 {\bf L}_{-i}(x) &\quad\mbox{if}\quad  
i\leq x < i+1 \quad\mbox{and}\quad n_0+3\leq i <n
\end{cases}$$
and $f^{\infty}_{R_{n+1}}(x) = 0$ otherwise.

\item If $n\geq 3$ is  an odd number then 
$$f_{R_p, n+1}(x)  = \begin{cases} 
  {\bf Z_{-i}}(x)& \quad\mbox{if}\quad 
 i\leq x < i+1 \quad\mbox{and}\quad 0\leq i \leq  n_0\\\\
 {\bf  Z}_{-n_0-1}(x) &\quad\mbox{if}\quad  
(n_0+1) \leq x < (n_0+\frac{3}{2})\\\\
{\bf L}_{-n_0-1}(x) & \quad\mbox{if}\quad  
(n_0+\frac{3}{2})\leq x < (n_0+2)\\\\
 {\bf L}_{-i}(x) &\quad\mbox{if}\quad  
i\leq x < i+1 \quad\mbox{and}\quad n_0+2\leq i <n
\end{cases}$$
and $f^{\infty}_{R_{n+1}}(x) = 0$ otherwise.
\end{enumerate}
\end{thm}

\begin{proof}The description of the function $f^{\infty}_{R_{n+1}}:[0, \infty)\longto 
[0, \infty)$ follows from Lemma~\ref{l4}.
To prove the symmetry, we consider the 
${\bf Z}_{-j}:[0, \infty)\longto [0, \infty)$ and 
${\bf L}_{-j}:[0, \infty)\longto [0\, \infty)$ 
and 
\vspace{5pt} 

\noindent{\bf Claim}.\quad ${\bf Z}_{-j}(x) = {\bf L}_{n-1-j}(n-x)$, if  $j\leq x < j+1$.

\vspace{5pt}

\noindent{\underline {Proof of the claim}}:\quad  By induction on $j\geq 0$, first 
we prove the assertion  that 
 $$\lim_{q\to \infty}Z^s_{-i}(a)/q^n = 
\lim_{q\to \infty}L^s_{-(n-1-i)}(q-a)/q^n\quad\mbox{for}\quad 0\leq a <q.$$

If $j=0$ then 
$$\lim_{q\to \infty}Z^s_{0}(a)/q^n = \lim_{q\to \infty}Y_a/q^n = 
\lim_{q\to \infty}Y_{a+n}/q^n = 
\lim_{q\to \infty}L^s_{-(n-1)}(q-a)/q^n.$$ 

Assume that the assertion holds for $0\leq j <i$.
Now 
$$\begin{array}{lcl}
\lim_{q\to \infty} Z_{-i}(a)/q^n & = & \lim_{q\to \infty}{Y_{a}}/{q^n}-
{\left[Y_1Z_{-i+1}(a)+\cdots+Y_iZ_0(a)\right]}/{q^n}\\\
 & = & \lim_{q\to \infty}{Y_{a+n}}/{q^n}-
{\left[Y_1L_{-(n-i)}(q-a)+\cdots+Y_iL_{-(n-1)}(q-a)\right]}/{q^n}\\\
& = & \lim_{q\to \infty} L_{-(n-1-i)}(q-a)/q^n.\end{array}$$

Now to prove the claim, it is enough to prove for $x = m/q$, where $m\in \Z_{\geq 0}$.
If  $j\leq x < j+1$ then $m = a+jq$, where $0\leq a<q$. Now 
$$\begin{array}{lcl}
{\bf Z}_{-j}(x) & = & \lim_{q\to \infty}Z^s_{-j}(m -jq)/q^n
= \lim_{q\to \infty}L^s_{-(n-1-j)}((j+1)q-m))/q^n\\\
& = & \lim_{q\to \infty}L^s_{-(n-1-j)}((n-m)q-(nq-q-jq))/q^n
 = {\bf L}_{-(n-1-j)}(n-x).\end{array}$$
This proves the claim.

\vspace{5pt}

\noindent{If} $n$ is even then $n_0 = n/2-1$.
Let  $0\leq x < (n-2)/2= n_0$ then $i \leq x <(i+1)$ for some $0\leq i\leq (n_0-1)$.
Now 
$$f^{\infty}_{R_{n+1}}(x) = {\bf Z}_{-i}(x) = {\bf L}_{-(n-1-i)}(n-x) = 
f^{\infty}_{R_{n+1}}(n-x).$$
 where the second equality follows as  
$n-(i+1) < n-x\leq n-i$.

\vspace{5pt} 

\noindent{If} $n$ is odd then $n_0 = (n-1)/2$.
Let $0\leq x <(n-2)/2= n_0-(1/2)$.

If $i\leq x <(i+1)$, where $i\leq n_0-1$ then 
$$f^{\infty}_{R_{n+1}}(x) = {\bf Z}_{-i}(x) = {\bf L}_{-(n-1-i)}(n-x) = 
f^{\infty}_{R_{n+1}}(n-x).$$
 If $(n_0-1)\leq x < n_0-1/2$ then again 
$$f^{\infty}_{R_{n+1}}(x) = {\bf Z}_{-(n_0-1)}(x) = {\bf L}_{-(n_0+1)}(n-x) = 
f^{\infty}_{R_{n+1}}(n-x).$$
\end{proof}

\begin{rmk}The same argument as above proves that $f_{R_p, n+1}$ is partially symmetric 
and the symmetry is given by 
$$f_{R_p, n+1}(x) = f_{R_p, n+1}(n-x)\quad \mbox{for}\quad 0\leq x \leq 
\frac{n-2}{2}\left(1-\frac{1}{p}\right).$$ 
\end{rmk}

\vspace{5pt}

\noindent{\underline {Proof of Theorem~\ref{t2}}}.\quad
If $n$ is even then $n_0 = \frac{n-2}{2}$ 
and the interval 
$$\left[n_0+2 -\frac{n-2}{2p},~~ n_0+2 +\frac{n-2}{2p}\right) = 
\left[\frac{n+2}{2}-\frac{n-2}{2p},~~\frac{n+2}{2}+\frac{n-2}{2p}\right)$$ 

If $n$ is odd then $n_0 = \frac{n-1}{2}$. and the interval 
$$\left[n_0+\frac{3}{2} -\frac{n-2}{2p},~~ n_0+\frac{3}{2} +\frac{n-2}{2p}\right) = 
\left[\frac{n+2}{2}-\frac{n-2}{2p},~~\frac{n+2}{2}+\frac{n-2}{2p}\right)$$

Note that, by Lemma~\ref{l4} and Theorem~\ref{f0}
$$f_{R_{p, n+1}}(x) = f_{R^{\infty}_{n+1}}(x)~~\mbox{ if}~~x\not\in 
\left[\frac{n+2}{2} -\frac{n-2}{2p},~~ \frac{n+2}{2} +\frac{n-2}{2p}\right).$$

Since both $f_{R_{p, n+1}}$ and $f_{R^{\infty}_{n+1}}$ are continuous 
functions on $\R$, it is enough to prove the rest of the assertion for $x\in \Z[1/p]$.
Now let $xq_0\in \Z$ 
for some $q_0 = p^{s_0}$. 

\begin{enumerate}
\item Let $n\geq 4$ be an even number with $n-2\leq p$.

\begin{enumerate}
\item Let $n_0+2 - \frac{n-2}{2p} \leq x < n_0+2$.
For  a fix $q \geq q_0$ let $a_q = xq -(n_0+1)q$. Then  
$0\leq a_q <q$ for all $q\geq q_0$ and by Lemma~\ref{l2}~(2)
$$\ell\left(\frac{R_{p, {n+1}}}{{\bf m}^{[q]}}\right)_{xq} = \nu^s_{-n_0-1}(a_q)+
2\lambda_0\mu^s_{-n_0}(a_q) 
= Z^s_{-n_0-1}+
2\lambda_0\mu^s_{-n_0+1}(a_q).$$
Hence $$f_{R_{p, n+1}}(x) = {\bf Z}_{-n_0-1}(x)+
\lim_{q\to \infty}2\lambda_0\frac{\mu^s_{-n_0+1}(a_q)}{q^n}$$
whereas
$f_{R^{\infty}_{n+1}}(x) = {\bf Z}_{-n_0-1}(x).$

\item  Let $n_0+2 \leq x < n_0+2 +\frac{n-2}{2p}$.
For a fix $q \geq q_0$ let $a_q = xq -(n_0+2)q$ then $0\leq a_q < q$ and, by 
Lemma~\ref{l2}~(5)
 $$\ell\left(\frac{R_{p, {n+1}}}{{\bf m}^{[q]}}\right)_{xq} = \nu^s_{-n_0-2}(a_q)+
2\lambda_0\mu^s_{-n_0-1}(a_q) 
= L^s_{-n_0-2}+
2\lambda_0\mu^s_{-n_0-1}(a_q).$$
Hence 
$$f_{R_{p, n+1}}(x) = {\bf L}_{-n_0-2}(x)+
\lim_{q\to \infty}2\lambda_0\frac{\mu^s_{-n_0-1}(a_q)}{q^n}$$
whereas 
$f_{R^{\infty}_{n+1}}(x) = {\bf L}_{-n_0-2}(x).$
\end{enumerate}

\item Let $n\geq 3$ be an odd number with $2n-4\leq p$.

\begin{enumerate}
\item Let $n_0+\frac{3}{2} - \frac{n-2}{2p} \leq x < n_0+\frac{3}{2} + \frac{n-2}{2p}$.

For  a fix $q = p^s\geq q_0$ let $a_q = xq -(n_0+1)q$. Then  
 $0\leq a_q <q$ and 
$$\begin{array}{lcl}
\ell\left(\frac{R_{p, {n+1}}}{{\bf m}^{[q]}}\right)_{xq} & = & \nu^s_{-n_0-1}(a_q)+
2\lambda_0\mu^s_{-n_0}(a_q) 
= Z^s_{-n_0-1}+
2\lambda_0\mu^s_{-n_0+1}(a_q)\\\
&  = &  L_{-n_0-1}(a_q)- 2\lambda_0\mu_{-n_0}(a_q),\end{array}$$
where the last equality follows as, by Lemma~\ref{l2}~(6)
and Lemma~\ref{l1}~(2)(b)(i)
$\nu^s_{-n_0-1}(a_q) = L_{-n_0-1}(a_q)$.

Hence we can write
$$\begin{array}{lcl}
f_{R_{p, n+1}}(x) & = & {\bf Z}_{-n_0-1}(x)+
\lim_{q\to \infty}2\lambda_0\frac{\mu^s_{-n_0+1}(a_q)}{q^n}\\\
& = & {\bf L}_{-n_0-1}(x)+
\lim_{q\to \infty}2\lambda_0\frac{\mu^s_{-n_0}(a_q)}{q^n}.\end{array}$$

Wheras 
$$f_{R^{\infty}_{n+1}}(x) = \begin{cases}
{\bf Z}_{-n_0-1}(x)\quad\mbox{if}\quad  n_0+\frac{3}{2} - \frac{n-2}{2p} \leq x 
< n_0+\frac{3}{2}\\\
{\bf L}_{-n_0-1}(x)\quad\mbox{if}\quad  n_0+\frac{3}{2}  \leq x 
< n_0+\frac{3}{2}+ \frac{n-2}{2p}.\end{cases}$$
\end{enumerate}
\end{enumerate}

This proves the theorem. $\Box$

\vspace{10pt}

\noindent{\underline {Proof of Theorem~\ref{t1}}}.\quad 
We note that, for any integer $0\leq a <q$ and $q=p^s$, we have 
the decomposition 
$$F_*^s(\sO(a)) =  \sum_{n_0-1}^0\sO(-i)^{\nu^s_{-i}(a)} \oplus\cdots 
\oplus\oplus \sum_{i = n_0-1}^{n_0+1}\sS(-i)^{\mu^s_{-i}(a)}.$$

By computing the ranks we get 
$q^n = \sum_{n_0-1}^0 \nu_{-i}^s(a)+ \sum_{i = n_0-1}^{n_0+1}\lambda_0\mu_{-i}^s(a)$.
In particular $0\leq \lambda_0{\mu_{-j}^s(a)}/{q^n}  
\leq 1$.

Therefore, by Lemma~\ref{l4} and  by the proof of Theorem~\ref{t2}
we have 
$$\begin{array}{l}
0\leq \int_0^\infty f_{R_{p, n+1}}(x)dx - \int_0^\infty f_{R^{\infty}_ {n+1}}(x)dx\\\\
 = \int_{\frac{n+2}{2}-\frac{n-2}{2p}}^{\frac{n+2}{2}+\frac{n-2}{2p}}
(f_{R_{p, n+1}}(x) - f_{R^{\infty}_ {n+1}}(x))dx \leq \frac{2n-4}{p}.\end{array}$$

On the other hand by Theorem~1.1 of [T] we have 
$$ e_{HK}({R_{p, n+1}}, {\bf m}) =   \int_0^{\infty} f_{{R_{p, n+1}}, {\bf m}}(x)dx.$$
This gives 
$$1+m_{n+1} =   \lim_{p\to \infty}e_{HK}({R_{p, n+1}}, {\bf m}) = 
\lim_{p\to \infty}\int_0^{\infty} f_{{R_p, n+1}, {\bf m}}(x)dx =  
\int_0^{\infty} f_{R^{\infty}_{n+1}, {\bf m}}(x)dx,$$
where  the first  equality follows by
 the result of Gessel-Monsky [GM], this can also be derived using Theorem~\ref{f0},
 in principle. $\Box$

\begin{cor}\label{c1}Let $p>2$.  
\begin{enumerate}
\item If $n$ even and $p\geq n-2$, or
\item if $n$ and $p\geq 2n-4$ 
\end{enumerate}
then the $F$-threshold of the ring
$R_{p ,n+1}$ is $c^{\bf m}({\bf m}) = n$.
\end{cor}
\begin{proof}  By Theorem~E of [TW1],  
the $F$-threshold 
$c^{\bf m}({\bf m}) = 
\mbox{max}~\{x\mid f_{R_{p, n+1}}(x)\neq 0\}$.

Now, by Lemma~\ref{l4}, $f_{R_{p, n+1}}(x) = 0$, for $x\geq n$ and for $n-1\leq x\leq n$, 
$$f_{R_{p, n+1}}(x) = {\bf L}_{-n+1}(x) = 
\lim_{q\to \infty}\frac{L^s_{-n+1}(\lfloor xq\rfloor-(n-1)q)}{q^n} = 
\frac{2(n-x)^n}{n!},$$
where the last equality follows as $L^s_{-n+1}(a) = Y_{q-a-n}$.
\end{proof}

\section{The HK density function for $R_{p, 4}$}
\begin{notations}\label{n4}
Let $p\geq 5$ be a prime and 
$$\mbox{$P_0 = \frac{p-1}{2}\quad{and}\quad 
P_i = \frac{p-1}{2p}\left[\frac{1}{p^{i-1}}+ \cdots+ \frac{1}{p}+ 
1\right]\quad{for}\quad i\geq 1$}.$$
then 
 $$\mbox{$P_1 < \cdots < P_j < P_{j+1} < \cdots <\frac{1}{2} <\cdots 
< \left(P_{j+1}+\frac{1}{p^{j+1}}\right) <  
\left(P_{j}+\frac{1}{p^{j}}\right) <\cdots < \left(P_1+\frac{1}{p}\right)$}.$$

We divide  the interval 
$$\mbox{$[2,3) = [2,~~ 2+ \frac{p-1}{2p})\cup 
 [2+ \frac{p-1}{2p}),~~2+ \frac{p+1}{2p})\cup [2+ \frac{p+1}{2p},~~ 3)$}, $$ 
 then $[2+ \frac{p-1}{2p},~~2+ \frac{p+1}{2p}) = [2+ P_1,~~ 2+ P_1+\frac{1}{p})$ 
can be further divided as

$$[2+ P_1,~~ 2+ P_1+\frac{1}{p}) =  \bigcup_{j=1}^{\infty} [2+P_{j},~~ 
2+P_{j+1})\cup \{2+\frac{1}{2}\}\cup\bigcup_{j=1}^{\infty} [2+P_{j+1}+
\frac{1}{p^{j+1}},~~2+P_{j}+\frac{1}{p^{j}}).$$ 

Let $$\mu_{-1} = \mu^1_{-1}(P_0-2)\quad\mbox{and}\quad {\overline {\mu_{-1}}} 
= {\tilde \mu^1_{-1}}(P_0-1),$$
 where the formula for 
 $\mu^1_{-1}(a)$ and ${\tilde \mu^1_{-1}}(a)$ is given in Lemma~\ref{F1}

\end{notations}

\begin{thm}\label{t3}Let $k$ be a perfect field of characteristic $p\geq 5$ and let   
$$R_{p, 4} = \frac{k[x_0, x_1. x_2, x_3, x_4]}{(x_0^2+x_1^2+ x_2^2 + x_3^2 + x_4^2)}.$$ 
Then 
$$f_{R_{p, 4}, {\bf m}}(x)  = \begin{cases} x^3/3 &  \quad\mbox{for}\quad  0\leq x < 1\\\\
  x^3/3 -5/3(x-1)^3 & \quad\mbox{for}\quad  1\leq x < 2\\\\
 \frac{1}{3}x^3 -\frac{5}{3}(x-1)^3+\frac{11}{3}(x-2)^3 & \quad\mbox{for}\quad  
2\leq x < 2+P_1\end{cases}$$

 $$\begin{array}{lcl}
f_{R_{p, 4}, {\bf m}}(x) & = &
  \frac{(3-x)^3}{3}+ 
\frac{4}{3} \sum_{i = 1}^{j} \left[\frac{1}{p^i}+P_i +2-x\right]^3
(\mu_{-1}{\overline {\mu_{-1}}}^{i-1})\\\\
& &  +\left[\frac{8}{3}\left[x-2-P_j\right]^3-
\frac{4}{p^j}\left[x-2-P_j\right]^2
+\frac{2}{3p^{3j}}\right]
({\overline {\mu_{-1}}}^j),\\\\
& &\hspace{10pt} \quad\mbox{for}\quad 2+P_j\leq x < 2+P_{j+1}\quad\mbox{and for}\quad j\geq 1\end{array}$$
 $$\begin{array}{lcl}
f_{R_{p, 4}, {\bf m}}(x) & = & 
 \frac{(3-x)^3}{3}+ 
\frac{4}{3} \sum_{i = 1}^{j} \left[\frac{1}{p^i}+P_i +2-x\right]^3
(\mu_{-1}{\overline {\mu_{-1}}}^{i-1}),\\\\
& &  \mbox{for}\quad 2+P_{j+1}+\frac{1}{p_{j+1}} \leq x < 2+P_{j}+\frac{1}{p_j}\quad\mbox{and for}\quad j\geq 1.\end{array}$$
$$f_{R_{p, 4}, {\bf m}}(x)   = \begin{cases}\frac{(3-x)^3}{3} &\quad\mbox{for}\quad 
2+P_{1}+\frac{1}{p_{1}} \leq x <3\\\\
0 & \quad\mbox{for}\quad x\geq 3.\end{cases}$$
\end{thm}
\begin{proof}Since we know that the function $f_{R_{p, 4}, {\bf m}}$ is continuous
 and the function on the right hand side is piecewise polynomial, 
it is enough to  prove the equality for the dense subset 
$\{m/p^l\mid l, m \in \Z_{\geq 0}\}$ 
 of $[0, \infty)$, For $q=p^s$ and $xq = m =a+iq$ where $0\leq a <q$
we have 
 $$f_{R_{p, 4}}(x) =  \lim_{s\to \infty}\frac{1}{p^{3s}}
\ell(R_{p, 4}/{\bf m}^{[q]})_{a+iq}$$

We fix $q =p^s$ and the nonnegative integer $a<q$. By 
 Lemmas~\ref{l2} and \ref{l3}, for $n_0=1$ we get 
$$ \begin{array}{lcll}
\ell (\frac{R_{p, 4}}{{\bf m}^{[q]}})_{a} & = & Z^s_{0}(a)  = Y_a & \quad\mbox{for}\quad 
 0\leq a < q\\\\
\ell (\frac{R_{p, 4}}{{\bf m}^{[q]}})_{a+q}& = & 
 Z^s_{-1}(a) = Y_{a+q}-Y_1Y_a &\quad\mbox{for}\quad  
0\leq a < q\\\\
\ell (\frac{R_{p, 4}}{{\bf m}^{[q]}})_{a+2q}& = & 
 Z^s_{-2}(a) = Y_{a+2q}-5Y_{a+q}+11Y_a &\quad\mbox{for}\quad  
0\leq a < \frac{q}{p}(\frac{p-1}{2})\\\\
\ell (\frac{R_{p, 4}}{{\bf m}^{[q]}})_{a+2q}& = & \nu^s_{-2}(a)+
2\lambda_0\mu^s_{-1}(a) & \quad\mbox{for}\quad  
\frac{q}{p}(\frac{p-1}{2})\leq a < \frac{q}{p}(\frac{p+1}{2}) \\\\
\ell (\frac{R_{p, 4}}{{\bf m}^{[q]}})_{a+2q}& = & L_{-2}^s(a) = Y_{q-a-3}
&\quad\mbox{for}\quad  
\frac{q}{p}(\frac{p+1}{2}) \leq a <  q\\\\
 & = & 0 & \quad\mbox{otherwise}.
\end{array}$$

By Lemma~\ref{l2}~(6) and (7) we have  $\nu_{-2}^s(a) = Y_{q-a-3}$. Therefore 
we only  need to compute 
$\mu_{-1}^s(a)$  for $a$ in the range
$\frac{p-1}{2p}\leq a/q < \frac{p+1}{2p}$.

We will use the following fact: If $b_0+
\cdots+b_{m-1}p^{m-1} = b$ is a $p$-adic expansion of $b$ then 
$$b_{m-1} < P_0 \iff b/p^m < (p-1)/2p\quad\mbox{and}\quad   
b_{m-1} > P_0 \iff  b/p^m \geq (p+1)/2p.$$ 
Therefore, by Lemma~\ref{l3}~(2)~(ii), 
$$b_{m-1} < P_0 \implies \mu_{-1}^m(b) = Z^m_{-2}(b)-\nu^m_{-2}(b)
=
Y_{b+2p^m}- Y_1Y_{b+p^m} + 
(Y_1^2-Y_2)Y_b -Y_{p^m-b-3}$$
and $b_{m-1} > P_0 \implies \mu_{-1}^m(b) = 0$.
Moreover for $m=1$, by Lemma~\ref{F1}, 
 $b = b_0 \geq  P_0$ implies $\mu_{-1}^1(b) = 0$.

Consider the $p$-adic expansion 
 $a_0+a_1p+\cdots + a_{s-1}p^{s-1}$ of $a$. Then by the hypothesis on $a$  we have
$a_{s-1} = P_0$.

In general 
if $1\leq j \leq s-1$ is an integer such that 
$a_{s-j} = \cdots = a_{s-1} =  P_0. $
then $P_j \leq a/q < 
P_j+\frac{1}{p^j}.$
Moreover
\begin{enumerate}
\item $a_{s-j-1} <  P_0 \iff P_j 
\leq a/q < P_{j+1}.$

\item $a_{s-j-1} = P_0 \iff P_{j+1} 
\leq a/q < P_{j+1}+\frac{1}{p^{j+1}}.$

\item $a_{s-j-1} > P_0 \iff P_{j+1}+
\frac{1}{p^{j+1}} 
\leq a/q < P_j +\frac{1}{p^{j}}.$
\end{enumerate}

We choose
\begin{enumerate}
\item  $j=s-1$ if $a_0 = a_1 = \cdots = a_{s-1} = P_0$. Otherwise 
\item  $1\leq j \leq s-1$ is the integer such that $a_{s-j} = \cdots = a_{s-1} = P_0$ and
$a_{s-j-1}\neq P_0$.
\end{enumerate}
We denote  $A_{s-i} = a_0+a_1p+\cdots + a_{s-i-1}p^{s-i-1}.$
Therefore 
$$A_{s-j} = a_0+a_1p+\cdots + a_{s-j-1}p^{s-j-1}m \quad\mbox{where}
\quad a_{s-j-1} \neq P_0.$$
Hence   $\mu^{s-j}_{-1}(A_{s-j})$
and, for all $i$, $\nu^{s-i}_{-2}(A_{s-i})$ are computable.
 By Lemma~\ref{F1},  the numebrs 
$\mu^1_{-1}(b)$ and ${\tilde \mu^1_{-1}}(b)$ are computable.

\vspace{5pt}

\noindent{\bf Claim}.\quad Let  $\mu_{-1} = \mu^1_{-1}(P_0-2)$
 and 
${\overline {\mu_{-1}}} = {\tilde \mu^1_{-1}}(P_0-1)$.
Then 
$$\begin{array}{lcl}
\mu_{-1}^s(a) &  = & \nu_{-2}^{s-1}(A_{s-1})(\mu_{-1}) + \nu_{-2}^{s-2}(A_{s-2})(\mu_{-1} 
{\overline {\mu_{-1}}})\\\\
 & & +\cdots + \nu_{-2}^{s-j}(A_{s-j})(\mu_{-1}{\overline {\mu_{-1}}}^{j-1})
+\mu_{-1}^{s-j}(A_{s-j})({\overline {\mu_{-1}}}^j).\end{array}$$
\vspace{5pt}

\noindent{\underline {Proof of the claim}}:\quad
For an integer $m$ and $q= p^s$, we have the decomposition (by [A]) 
 $$F^s_*(\sO(m)) =  \sO(-2)^{\nu^s_{-2}(m)}\oplus
\sO(-1)^{\nu^s_{-1}(m)}\oplus \sO^{\nu^s_0(m)}\oplus M(0)^{\mu^s_{0}(m)}\oplus 
M(-1)^{\mu^s_{-1}}(m)$$
 $$F_*^{s}(\sS(m))=   \sO(-2)^{{\tilde \nu}^{s}_{-2}(m)}\oplus
\sO(-1)^{{\tilde \nu}^s_{-1}(m)}\oplus \sO^{{\tilde \nu}^s_0(m)}
\oplus M(0)^{{\tilde \mu}^s_{0}(m)}\oplus 
M(-1)^{{\tilde \mu}^s_{-1}(m)}.$$

By the projection formula  
$$F^s_*(\sO(a)) = F^j_*(F_*^{s-j}(\sO(A_{s-j}))\tensor \sO(P_0+\cdots 
+P_0p^{j-1})).$$ Therefore

\begin{equation}\label{***}[\nu^s_{-2}(a), \nu^s_{-1}(a), \nu^s_{0}(a),
\mu^s_{0}(a),\mu^s_{-1}(a)]
=  [\nu^{s-j}_{-2}(A_{s-j}), \cdots, \mu^{s-j}_{-1}(A_{s-j})]\cdot 
[b_{kl}]\times\mbox{j times}\times [b_{kl}],
\end{equation}
where $[b_{kl}]$ is the matrix  
$$[b_{kl}] = \left[\begin{matrix}
\nu^1_{-2}(P_0-2) & \nu^1_{-1}(P_0-2)  & \nu^1_{0}(P_0-2) & 0 & \mu_{-1}\\\\ 
\nu^1_{-2}(P_0-1) & \nu^1_{-1}(P_0-1) & \nu^1_{0}(P_0-1) & 0 & 0\\\\
\nu^1_{-2}(P_0) & \nu^1_{-1}(P_0)  & \nu^1_{0}(P_0) &\mu^1_{0}(P_0) & 0 \\\\ 
{\tilde \nu^1_{-2}}(P_0) & {\tilde \nu^1_{-1}}(P_0) & {\tilde \nu^1_{0}}(P_0) &
{\tilde \mu^1_{0}}(P_0) & 0\\\\
{\tilde \nu^1_{-2}}(P_0-1) & {\tilde \nu^1_{-1}}(P_0-1) & {\tilde \nu^1_{0}}(P_0-1) & 
0  & {\overline {\mu_{-1}}}.\end{matrix}\right]$$
Now the claim follows by induction on $j$.

If $a_0 = \cdots = a_{s-1} = P_0$ then $A_{s-j} = a_0$ and 
$\mu_{-1}^1(a_0) = 0$.

We recall that 
$$Y_a= \frac{1}{6}(2a^3+9a^2+13a+6) = a^3/3+O(a^2).$$
Hence 
$$\lim_{s\to \infty} \frac{\nu_{-2}^{s-i}(A_{s-i})}{p^{3s}} =
\lim_{s\to \infty} \frac{Y_{p^{s-i}-(a-p_0(p^{s-i}+\cdots + p^{s-1})-3}}{p^{3s}} = 
\frac{1}{3}\left[\frac{1}{p^i}+P_i-x\right]^3.$$

Now

\begin{enumerate}
\item If there is $1\leq j \leq s-1$ such that  
$P_j \leq a/q < P_{j+1}$ then 
$a_{s-j-1} < P_0$ and 
$$\lim_{q\to \infty} \frac{(4)\mu_{-1}^{s-j}(A_{s-j})}{q^3} = 
\lim_{q\to \infty} Z_{-2}^{s-j}(A_{s-j})-\nu_{-2}^{s-j}(A_{s-j})
 = \frac{8}{3}\left[x-P_j\right]^3-
\frac{4}{p^j}\left[x-P_j\right]^2
+\frac{2}{3p^{3j}}.$$

Hence 
$$\begin{array}{lcl}
\lim_{q\to \infty}\frac{\nu_{-2}^s(a) + 4\mu_{-1}^s(a)}{q^3} & = &
\frac{(1-x)^3}{3}+ 
\frac{4}{3} \sum_{i = 1}^{j} \left[\frac{1}{p^i}+P_i -x\right]^3
(\mu_{-1}{\overline {\mu_{-1}}}^{i-1})\\\\
& & +\left(\frac{8}{3}\left[x-P_j\right]^3-\frac{4}{p^j}\left[x-P_j\right]^2
+\frac{2}{3p^{3j}}\right)
({\overline {\mu_{-1}}}^j).\end{array}$$

\item If there is $1\leq j \leq s-1$ such that 
$P_{j+1}+\frac{1}{p^{j+1}} 
\leq a/q < P_j+\frac{1}{p^{j}}$. Then 
 $a_{s-j-1} > P_0$ and  hence $\mu_{-1}^{s-j}(A_{s-j}) = 0$.
This gives 
$$\lim_{q\to \infty}\frac{\nu_{-2}^s(a) + 4\mu_{-1}^s(a)}{q^3} = \frac{(1-x)^3}{3}+ 
\frac{4}{3} \sum_{i = 1}^{j} 
\left[\frac{1}{p^i}+
P_i - x\right]^3
(\mu_{-1}{\overline {\mu_{-1}}}^{i-1}).$$
 
\item If there is no $j$ satisfying the any of the above two cases then 
 $a/q = P_s$ and
$j = s-1$ and $A_{s-j} = A_1 = P_0$. But  $\mu_{-1}^1(P_0) = 0.$
Hence 
$$\lim_{q\to \infty}\frac{\nu_{-2}^s(a) + 4\mu_{-1}^s(a)}{q^3} = \frac{(1-x)^3}{3}+ 
\frac{4}{3} \sum_{i = 1}^{s-1} 
\left[\frac{1}{p^i}+
P_i-x\right]^3
(\mu_{-1}{\overline {\mu_{-1}}}^{i-1}).$$
\end{enumerate}

This proves the theorem.
\end{proof}

\end{document}

\bibitem[1]{1}{Aberbach, I.M.} {\it Extensions of weakly and strongly
$F$-rational rings by flat maps}, J. Algebra 241 (2001), 799-807.\\

\bibitem[2]{2}{Bourbaki, N.} {\it Commutative algebra}, Chapters 1–7, Elements of 
Mathematics (Berlin),Springer-Verlag, Berlin, 1998, 
Translated from the French, Reprint of the 1989 English translation. 
MR 1727221 (2001g:13001)\\

\bibitem[3]{3}{ W. Bruns and J. Herzog}, {\it Cohen-Macaulay rings}, 
Cambridge Stud. Adv. Math., Vol 39, Cambridge University Press, 
Cambridge, 1993.\\

\bibitem[4]{4}{N\'{u}\~{n}ez-Betancourt, L., P\'{e}rez, F., Stefani, A.}, {\it
On the existence of F-thresholds and related limits}, Trans. Amer. Math. Soc.
370 (2018), no. 9, 6629-6650.\\

\bibitem[5]{5}{M. Demazure}, {\it Anneaux gradu\'{e}s normaux}, in Seminaire 
Demazure-Giraud-Teissier, Singularities des surfaces, Ecole Polytechnique, 1979.\\

\bibitem[7]{7}{R. Hartshorne,}, {\it Stable reflexive sheaves}, Math. Ann.
{\bf 254} (1980), no. 2, 121-176.\\

\bibitem[8]{8}{Hochster, M., Huneke, C.} {\it Tight closure,
invariant theory and the Brian\c{c}on-Skoda theorem}, J. Amer. Math. Soc. 3 (1990),
31-116.\\

\bibitem[9]{9}{Hochster, M., Huneke, C.} {\it $F$-regularity, test
elements, and smooth base change}, Trans. Amer. Math. Soc. 346 (1994), 
no. 1, 1-62.\\

\bibitem[10]{10}{C. Huneke, M. Musta\c{t}\u{a}, S. Takagi, K.I. Watanabe},
{\it F-thresholds, tight closure, integral closure and multiplicity bounds},
 Michigan Math. J. 57, in Special Volume in  Honor  of  Melvin  Hochster,
Univ.  Michigan  Press,  Ann  Arbor,  2008,  463–483.\\

\bibitem[11]{11}{C. Huneke, S. Takagi and  K.I. Watanabe}, {\it 
Multiplicity bounds in graded rings},  Kyoto J. Math. 51 (2011), no. 1, 127-147.\\

\bibitem[12]{12}{K. Kurano}, {\it The singular Riemann-Roch 
theorem and Hilbert-Kunz functions}, J. Algebra, 304, 487-499.\\

\bibitem[13]{13}{G.A. Miller, H.F. Blichfeldt and L.E. Dickson}, {\it 
Theory and applications of finite groups}, Reprint of 1916 Edition with 
corrections, New York, G.F Stechert $\&$ Co., 1938.\\

\bibitem[15]{15}{Musta\c{t}\u{a}, M., Takagi, S., Watanabe, K.I.},
{\it F-thresholds and Bernstein-Sato polynomials}, European
congress of mathematics, 341-364, Eur. Math. Soc., Zurich, 2005.\\

\bibitem[16]{16}{C. Peskine and L. Szpiro}, {\it Dimension projective finie
et cohomologie locale}, Publ. Math. I.H.E.S. 42 (1972), 47-119.\\

\bibitem[17]{17}{K. Schwede}, {\em Generalised divisors and reflexive sheaves}, 
Homepage.\\

\bibitem[18]{18}{M. Tomari}, {\it Multiplicity of filtered rings and simple 
$K3$ singularities of multiplicity two}, Publ. Res. Inst. Math. Sci. {\bf 38} (2002), 
693-724.\\

\bibitem[21]{21}{Watanabe, K.I.}, {\it Some remarks concerning Demazure’s construction of 
normal graded rings}, Nagoya Math. J.83(1981), 203-211.\\

\bibitem[22]{22}{K.I. Watanabe and K.I. Yoshida} {\it Hilbert-Kunz 
multiplicity and an inequality between multiplicity and 
colength}, J. Algebra, {\bf 230}  (2000), 295-317


\begin{thebibliography}{}

\bibitem[AE1]{AE1}{I. M. Aberbach and F. Enescu}, {\it Lower bounds for Hilbert-Kunz 
multiplicities in local rings  of  fixed  dimension},  Michigan  Math.  J.57 (2008),  
1-16.\\  

\bibitem[AE2]{AE2}{I. M. Aberbach and F. Enescu}, {\it New estimates of Hilbert-Kunz 
multiplicities for local rings of fixed dimension}, Nagoya Math. Journal, 
{\bf 212} (2013), 59-85.\\

\bibitem[A]{A}{P. Achinger}, {\it Frobenius push-forwards on quadrics}. Comm. Algebra 40 (2012), no. 8, 2732-2748.\\ 

\bibitem[BEH]{BEH}{R. Buchweitz, D. Eisenbud and J. Herzog}, {\it  Cohen–Macaulay 
modules on quadrics . In: Singularities, Representation of Algebras, and Vector Bundles}
 (Lambrecht, 1985) . Vol. 1273 of Lecture Notes in Math . Berlin : Springer , pp. 58-116.\\

\bibitem[CDHZ]{CDHZ}, {O. Celikbas, H. Dao, C. Huneke and Y. Zhang},
{\it Bounds on the Hilbert-Kunz multi-plicity}, Nagoya Math. J.205(2012), 149-165.\\ 


\bibitem[E]{E}{D. Eisenbud}, {\it Homological algebra on a complete intersection, 
with an application to group representations}. Trans. Amer. Math. Soc. 260 (1): 35-64.\\

\bibitem[ES]{ES}{F.  Enescu  and  K.  Shimomoto}, {\it On  the  upper  semi-continuity  
of  the  Hilbert-Kunz multiplicity}, Journal of  Algebra 285 (2005), 222-237.\\

\bibitem[EY]{EY}{K. Eto and K.I. Yoshida}, {\it Notes on Hilbert-Kunz multiplicity of 
Rees algebras}, Comm. Algebra {\bf 31} (2003), no. 12, 5943-5976.\\


\bibitem[JNSWY]{JNSWY}{J. Jeffries, Y. Nakajima, I. Smirnov, K.I. Watanabe, K.I. Yoshida},
{\it Lower bounds on Hilbert--Kunz multiplicities and maximal F-signatures},
arXiv:2002.06166.\\

\bibitem[L]{L}{A. Langer}, {\it  D-affinity and Frobenius morphism on quadrics},
 Int. Math. Res. Not. IMRN (1) 145, 26 (2008).\\

\bibitem[H]{H}{R. Hartshorne,}, {\it Algebraic geometry}, Springer 1977.\\


\bibitem[M]{M}{P. Monsky}, {\it The Hilbert-Kunz function},
Math. Ann. 263 (1983) 43-49.\\

\bibitem[MG]{MG}{P. Monsky and  I. Gessel}, {\it The limit as $p\to \infty$ of the 
Hilbert-Kunz multiplicity of $\sum x_i^{d_i}$}, preprint, 2010, arXiv:1007.2004[math.AC].\\


\bibitem[T]{T}{V. Trivedi}, {\it Hilbert-Kunz Density Function and 
Hilbert-Kunz Multiplicity}, Trans. Amer. Math. Soc. 
370 (2018), no. 12, 8403-8428.\\

\bibitem[TW1]{TW1}{V. Trivedi and K.I. Watanabe}, {\it Hilbert-Kunz Density Functions and 
$F$-thresholds}, to appear in Journal of Algebra.\\


\bibitem[TW2]{TW2}{V. Trivedi and K.I. Watanabe}, {\it Hilbert-Kunz density 
function for  graded domains}, arXiv:2003.07035.\\


\bibitem[WY1]{WY1}{K.i. Watanabe and K.i. Yoshida}, {\it Hilbert-Kunz multiplicity 
and an  inequality between multiplicity and colength}, Journal of Algebra (2000), 
230, 295-317.\\


\bibitem[WY2]{WY2}{K.i. Watanabe and K.i. Yoshida}, {\it Hilbert-Kunz multiplicity 
of two-dimensional local rings}, Nagoya Math. Journal 162 (2001), 87-110.\\

\bibitem[Y]{Y}{K.i. Yoshida}, {\it Small Hilbert-Kunz multiplicity 
and $(A_1)$-type singularity} in Proceedings of the $4^{th}$ Japan-Vietnam Joint Seminar on Commutative Algebra by and for Young Mathematicians, Meiji University, Japan, 2009.


\end{thebibliography}
\end{document}